\documentclass{amsart}
\usepackage{graphicx,anysize} % Required for inserting images

\usepackage{hyperref,tikz,pgfplots,tikz-cd,amssymb}

\theoremstyle{plain}
\newtheorem{thm}{Theorem}[section]
\newtheorem{theorem}[thm]{Theorem}
\newtheorem{lemma}[thm]{Lemma}
\newtheorem{example}[thm]{Example}

\newtheorem{definition}[thm]{Definition}
\newtheorem{remark}[thm]{Remark}

\newcommand{\Gr}{\mathrm{Gr}}
\newcommand{\C}{\mathbb{C}}
\newcommand{\HHH}{\mathrm{HHH}}
\newcommand{\CF}{\mathcal{F}}
\newcommand{\sort}{\mathrm{sort}}

\tikzset
{%https://tex.stackexchange.com/questions/620408/how-can-i-draw-2-tikz-rectangular-grids-side-by-side
  pics/matrix/.style n args={5}{% rows, columns, text above, text below, text left (or symbol)
    code={%
      \begin{scope}[y=-1cm,scale=0.5]
        \draw    (0,0) grid (#2,#1);
        \node at (0.5*#2,-0.5)   {#3};
        \node at (0.5*#2,#1+0.5) {#4};
        \node at (-0.5,0.5*#1)   {#5};
      \end{scope}     
    }},
}

\title{Splicing positroid varieties}
\author{Eugene Gorsky}
\address{Department of Mathematics\\ University of California at Davis\\ One Shields Avenue, Davis CA 95616}
\email{egorskiy@math.ucdavis.edu}
\author{Tonie Scroggin}
\address{Department of Mathematics\\ University of California at Davis\\ One Shields Avenue, Davis CA 95616}
\email{tmscroggin@ucdavis.edu}
\date{}

\begin{document}

\begin{abstract}
We construct an explicit isomorphism between an open subset in the open positroid variety  $\Pi_{k,n}^{\circ}$ in the Grassmannian $\Gr(k,n)$ and the product of two open positroid varieties $\Pi_{k,n-a+1}^{\circ}\times \Pi_{k,a+k-1}^{\circ}$. In the respective cluster structures, this isomorphism is given by freezing a certain subset of cluster variables and applying a cluster quasi-equivalence.
\end{abstract}

\maketitle

\section{Introduction}

In this paper, we study certain maps between the open positroid varieties. Recall that the Grassmannian $\Gr(k,n)$ can be realized as the space of  $k\times n$ matrices of rank $k$, modulo row operations. We will denote the columns of such a matrix $V$ by $(v_1,\ldots,v_n)$, and extend them periodically by $v_{i+n}=v_i$. The {\bf open positroid variety} $\Pi_{k,n}^{\circ}$ is then defined \cite{KLS} by the condition that the $k$ cyclically consecutive vectors $v_i,v_{i+1},\ldots,v_{i+k-1}$ are linearly independent for all $i$. Equivalently, the minors $\Delta_{i,i+1,\ldots,i+k-1}(V)$ do not vanish. In the last decades, open positroid varieties (and more general positroid strata in $\Gr(k,n)$) have attracted a lot of attention in combinatorics, algebra, geometry and physics. We focus on two results which motivated this paper.

First, by the work of Scott \cite{FWZ,S} open positroid varieties carry a {\bf cluster structure} which we review in Section \ref{sec: cluster Scott}. In particular, there is a certain collection of minors of $V$ providing cluster coordinates on $\Pi_{k,n}^{\circ}$. Many more cluster coordinates are obtained from these by mutations. 

Second, by the work of Galashin and Lam \cite{GL} the torus-equivariant homology of $\Pi_{k,n}^{\circ}$ is closely related to the {\bf Khovanov-Rozansky homology} $\HHH$ of the torus link $T(k,n-k)$ . In fact, $\Pi_{k,n}^{\circ}$  (up to a certain algebraic torus) is isomorphic to a braid variety in the sense of \cite{CGGS1,CGGS2}, and by \cite{Trinh} the torus-equivariant homology of any braid variety is related  to the Khovanov-Rozansky homology of the corresponding link. Furthermore, by \cite{CGGLSS,GLSBS} any braid variety admits a cluster structure. 

The multiplication of braids $T(k,s)\cdot T(k,t)\to T(k,s+t)$ induces a natural map in Khovanov-Rozansky homology
$$
\HHH(T(k,s))\otimes \HHH(T(k,t))\to \HHH(T(k,s+t)).
$$
This suggests that there might be a map of open positroid varieties:
\begin{equation}
\label{eq: s and t}
\Pi_{k,k+s}^{\circ}\times \Pi_{k,k+t}^{\circ}\to \Pi_{k,k+s+t}^{\circ}.
\end{equation}
In this paper, we construct such a map with particularly nice properties and describe its image.

Given a $k\times n$ matrix $V=(v_1,\ldots,v_{n})$ as above and $1\le a\le n-k$, we consider a submatrix $V_1=(v_a,\ldots,v_n)$. This submatrix belongs to $\Pi_{k,n-a+1}^{\circ}$ if and only if certain additional minors of $V$ do not vanish, which defines an open subset $U_a\subset \Pi_{k,n}^{\circ}$. On this open subset, we define the second matrix 
$$
V_2=(v_1,\ldots, v_{a},u_1,\ldots,u_{k-2},v_n)
$$
where $u_i$ are certain vectors \eqref{eq: def u} spanning the one-dimensional intersection (see Lemma \ref{lem: intersection}):
$$
\langle v_a,v_{a+1},\ldots,v_{a+i}\rangle\cap \langle v_{n-k+i+1},\ldots,v_{n-1},v_n\rangle=\langle u_i\rangle.
$$
Diagrammatically, we can visualize the rule for computing $u_i$ in Figure \ref{fig: 5 10 intro}. The diagram represents a specific  braid word for the half-twist braid, thought of as a subword of a longer torus braid (see Figure \ref{fig:braid cut 3,8}). The regions in the diagram are labeled by certain subspaces in $\C^k$ such that every vertical cross-section forms a complete flag and two neighboring flags have a prescribed relative position. For example, in Figure \ref{fig: 5 10 intro} we get  complete flags
$$
0\subset \langle v_3\rangle \subset \langle v_3,v_4\rangle\subset  \langle v_3,v_4,v_5\rangle
\subset  \langle v_3,v_4,v_5,v_6\rangle\subset \C^5
$$
on the left and 
$$
0\subset \langle v_{10}\rangle \subset \langle v_9,v_{10}\rangle\subset  \langle v_8,v_9,v_{10}\rangle
\subset  \langle v_7,v_8,v_9,v_{10}\rangle\subset \C^5
$$
on the right which determine all intermediate subspaces. See Section \ref{sec: braids} for more details. Curiously, this figure resembles the process of mRNA splicing \cite[Figure 2]{splicing}.

\begin{figure}[ht!]
    \centering
    \tikzset{every picture/.style={line width=0.75pt}} %set default line width to 0.75pt        

\begin{tikzpicture}[x=0.75pt,y=0.75pt,yscale=-1,xscale=1]

%Straight Lines [id:da43081404952437174] 
\draw    (120.13,170.42) -- (180,170.42) ;
%Straight Lines [id:da860114993418754] 
\draw    (120.13,140.42) -- (210.25,139.92) ;
%Straight Lines [id:da440518013726922] 
\draw    (120.13,110.17) -- (240.25,110.42) ;
%Straight Lines [id:da521301808874939] 
\draw    (120.13,200.42) -- (150,200.17) ;
%Straight Lines [id:da7314439986107959] 
\draw    (119.88,230.17) -- (150,230.17) ;
%Straight Lines [id:da18667523098749994] 
\draw    (150,200.17) -- (170.25,229.92) ;
%Straight Lines [id:da8363625573031155] 
\draw    (150,230.17) -- (158.75,217.17) ;
%Straight Lines [id:da7814998766989167] 
\draw    (161.75,213.17) -- (169.75,200.42) ;
%Straight Lines [id:da8305055603193974] 
\draw    (180,170.42) -- (200.25,200.17) ;
%Straight Lines [id:da6774550780463704] 
\draw    (180,200.42) -- (188.75,187.42) ;
%Straight Lines [id:da6162696947327795] 
\draw    (191.75,183.42) -- (199.75,170.67) ;
%Straight Lines [id:da44884704622762994] 
\draw    (210.25,139.92) -- (230.5,169.67) ;
%Straight Lines [id:da7507289769076497] 
\draw    (210,170.67) -- (219,156.92) ;
%Straight Lines [id:da5750670591362803] 
\draw    (222,152.92) -- (230,140.17) ;
%Straight Lines [id:da3420912754354959] 
\draw    (240.25,110.42) -- (260.5,140.17) ;
%Straight Lines [id:da6303815720809545] 
\draw    (240.25,140.42) -- (249,127.42) ;
%Straight Lines [id:da00664484093595008] 
\draw    (252,123.42) -- (260,110.67) ;
%Straight Lines [id:da12293138358217881] 
\draw    (270.25,200.17) -- (290.5,229.92) ;
%Straight Lines [id:da3531186658767105] 
\draw    (270.25,230.17) -- (279,217.17) ;
%Straight Lines [id:da21541171624363797] 
\draw    (282,213.17) -- (290,200.42) ;
%Straight Lines [id:da8592888424023577] 
\draw    (300.25,170.17) -- (320.5,199.92) ;
%Straight Lines [id:da9538275072722171] 
\draw    (300.25,200.42) -- (309,187.17) ;
%Straight Lines [id:da1188496582775056] 
\draw    (312,183.17) -- (320,170.42) ;
%Straight Lines [id:da7470305306011495] 
\draw    (330.25,140.42) -- (350.5,170.17) ;
%Straight Lines [id:da4724737792643683] 
\draw    (330.25,170.42) -- (339,157.42) ;
%Straight Lines [id:da820264230580249] 
\draw    (342,153.42) -- (350,140.67) ;
%Straight Lines [id:da09561466808043906] 
\draw    (360.5,199.92) -- (380.75,229.67) ;
%Straight Lines [id:da41865092057924946] 
\draw    (360.5,229.92) -- (369.25,216.92) ;
%Straight Lines [id:da8009826831322315] 
\draw    (372.25,212.92) -- (380.25,200.17) ;
%Straight Lines [id:da10675803389716676] 
\draw    (390.25,170.42) -- (410.5,200.17) ;
%Straight Lines [id:da2490164329830915] 
\draw    (390.25,200.42) -- (399,187.42) ;
%Straight Lines [id:da4756099419509221] 
\draw    (402,183.42) -- (410,170.67) ;
%Straight Lines [id:da7775015032357266] 
\draw    (419.75,200.17) -- (440,229.92) ;
%Straight Lines [id:da5359178484754723] 
\draw    (419.75,230.17) -- (428.5,217.17) ;
%Straight Lines [id:da7013910998171433] 
\draw    (431.5,213.17) -- (439.5,200.42) ;
%Straight Lines [id:da9341299252924127] 
\draw    (169.75,200.42) -- (180,200.42) ;
%Straight Lines [id:da7271773704657327] 
\draw    (199.75,170.67) -- (210,170.67) ;
%Straight Lines [id:da2577040327200544] 
\draw    (230,140.17) -- (240.25,140.42) ;
%Straight Lines [id:da2622275519447739] 
\draw    (290,200.42) -- (300.25,200.42) ;
%Straight Lines [id:da15448483353228415] 
\draw    (320,170.42) -- (330.25,170.42) ;
%Straight Lines [id:da23631163115163267] 
\draw    (380.25,200.17) -- (390.5,200.17) ;
%Straight Lines [id:da8665178576273245] 
\draw    (410.5,200.17) -- (419.75,200.17) ;
%Straight Lines [id:da41608868579923364] 
\draw    (170.25,229.92) -- (270.25,230.17) ;
%Straight Lines [id:da450686844124742] 
\draw    (200.25,200.17) -- (270.25,200.17) ;
%Straight Lines [id:da41636902114130936] 
\draw    (230.5,169.67) -- (300.25,170.17) ;
%Straight Lines [id:da11959811607656179] 
\draw    (260.5,140.17) -- (330.25,140.42) ;
%Straight Lines [id:da057556342319579956] 
\draw    (290.5,229.92) -- (360.5,229.92) ;
%Straight Lines [id:da48805180011896043] 
\draw    (320.5,199.92) -- (360.5,199.92) ;
%Straight Lines [id:da42242196465294546] 
\draw    (380.75,229.67) -- (419.75,230.17) ;
%Straight Lines [id:da7302288135841646] 
\draw    (350.5,170.17) -- (390.25,170.42) ;
%Straight Lines [id:da6440989444748026] 
\draw    (440,229.92) -- (470.25,229.92) ;
%Straight Lines [id:da9620238841063735] 
\draw    (439.5,200.42) -- (469.75,200.42) ;
%Straight Lines [id:da38502237710150444] 
\draw    (410,170.67) -- (470.25,170.42) ;
%Straight Lines [id:da7841016819134863] 
\draw    (350,140.67) -- (470,140.67) ;
%Straight Lines [id:da7657279874013829] 
\draw    (260,110.67) -- (470,110.42) ;

% Text Node
\draw (129.17,209.01) node [anchor=north west][inner sep=0.75pt]  [font=\footnotesize] [align=left] {$\displaystyle v_{3}$};
% Text Node
\draw (444.83,209.34) node [anchor=north west][inner sep=0.75pt]  [font=\footnotesize] [align=left] {$\displaystyle v_{10}$};
% Text Node
\draw (394.5,209.34) node [anchor=north west][inner sep=0.75pt]  [font=\footnotesize] [align=left] {$\displaystyle u_{3}$};
% Text Node
\draw (318.83,209.34) node [anchor=north west][inner sep=0.75pt]  [font=\footnotesize] [align=left] {$\displaystyle u_{2}$};
% Text Node
\draw (208.83,209.01) node [anchor=north west][inner sep=0.75pt]  [font=\footnotesize] [align=left] {$\displaystyle u_{1}$};
% Text Node
\draw (122.83,179.34) node [anchor=north west][inner sep=0.75pt]  [font=\footnotesize] [align=left] {$\displaystyle \langle v_{3} ,v_{4} \rangle $};
% Text Node
\draw (418.5,177.67) node [anchor=north west][inner sep=0.75pt]  [font=\footnotesize] [align=left] {$\displaystyle \langle v_{9} ,v_{10} \rangle $};
% Text Node
\draw (123.5,148.01) node [anchor=north west][inner sep=0.75pt]  [font=\footnotesize] [align=left] {$\displaystyle \langle v_{3} ,v_{4} ,v_{5} \rangle $};
% Text Node
\draw (122.83,119.01) node [anchor=north west][inner sep=0.75pt]  [font=\footnotesize] [align=left] {$\displaystyle \langle v_{3} ,v_{4} ,v_{5} ,v_{6} \rangle $};
% Text Node
\draw (400.83,147.34) node [anchor=north west][inner sep=0.75pt]  [font=\footnotesize] [align=left] {$\displaystyle \langle v_{8} ,v_{9} ,v_{10} \rangle $};
% Text Node
\draw (381.5,117.34) node [anchor=north west][inner sep=0.75pt]  [font=\footnotesize] [align=left] {$\displaystyle \langle v_{7} ,v_{8} ,v_{9} ,v_{10} \rangle $};
% Text Node
\draw (223,178.67) node [anchor=north west][inner sep=0.75pt]  [font=\footnotesize] [align=left] {$\displaystyle \langle u_{1} ,u_{2} \rangle $};
% Text Node
\draw (337,178.67) node [anchor=north west][inner sep=0.75pt]  [font=\footnotesize] [align=left] {$\displaystyle \langle u_{2} ,u_{3} \rangle $};
% Text Node
\draw (249,147.67) node [anchor=north west][inner sep=0.75pt]  [font=\footnotesize] [align=left] {$\displaystyle \langle u_{1} ,u_{2} ,u_{3} \rangle $};

\end{tikzpicture}
    \caption{Braid diagram and flags for $k=5,n=10$ and $a=3$. Here $\langle u_1\rangle =\langle v_3,v_4\rangle \cap \langle v_7,v_8,v_9,v_{10}\rangle$,
    $\langle u_2\rangle =\langle v_3,v_4,v_5\rangle \cap \langle v_8,v_9,v_{10}\rangle$ and 
    $\langle u_3\rangle =\langle v_3,v_4,v_5,v_6\rangle \cap \langle v_9,v_{10}\rangle$.}
    \label{fig: 5 10 intro}
\end{figure}
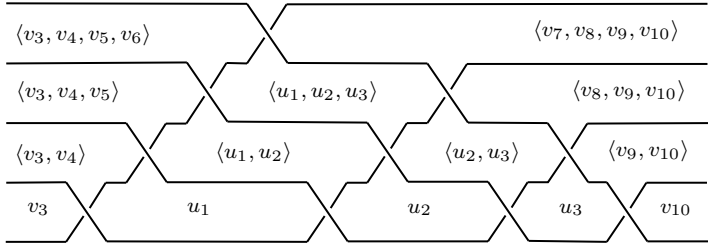

We can now state our main result.

\begin{theorem}
\label{thm: intro main}
Choose an integer $a$ such that $1\le a\le n-k$.

a) The map $\Phi_a:V\mapsto (V_1,V_2)$ is well defined and yields an isomorphism $$U_a\simeq \Pi_{k,n-a+1}^{\circ}\times \Pi_{k,a+k-1}^{\circ}.$$

b) The subset $U_a$ admits a cluster structure obtained by freezing certain explicit cluster variables in the rectangle seed for $\Pi_{k,n}^{\circ}$.

c) The map $\Phi_a$ is a cluster quasi-isomorphism (in the sense of \cite{F}) between $U_a$ and $\Pi_{k,n-a+1}^{\circ}\times \Pi_{k,a+k-1}^{\circ}$.
\end{theorem}

\begin{remark}
For $a=1$ we have $V_1=V\in \Pi^{\circ}_{k,n}$ and $V_2\in \Pi_{k,k}^{\circ}=\mathrm{pt}$, so $\Phi_1$ is the identity map. Most of the time, we will assume $2\le a\le n-k$ so that $\Phi_1$ is nontrivial.
\end{remark}

Note that the inverse map $\Phi_a^{-1}:\Pi_{k,n-a+1}^{\circ}\times \Pi_{k,a+k-1}^{\circ}\to \Pi_{k,n}^{\circ}$ has image $U_a$ and fulfills our expectation \eqref{eq: s and t} with $s=n-a+1-k$ and $t=a-1$. We give an interpretation of $\Phi_a$ in terms of braid varieties in Section \ref{sec: braids}.
In a future work, we plan to use the explicit description of the map $\Phi_a$ to study a more explicit relation between the homology of $\Pi_{k,n}^{\circ}$ and link homology. For example, in the case $k=2$ there are no vectors $u_i$, and the map $\Phi_a$ has the form
$$
(v_1,\ldots,v_n)\mapsto  (v_a,\ldots,v_n),(v_1,\ldots,v_a,v_n) .
$$
Such a map and its generalizations were studied in detail by the second author in \cite{Scroggin}, where the corresponding maps on de Rham cohomology were described explicitly.

Finally, we would like to notice a striking similarity between our map and the so-called BCFW recursion in theory of scattering amplitudes. In \cite{ELPTSBW} this recursion was reinterpreted as a rational map (see \cite[Definition 4.2]{ELPTSBW}) $\Psi:\Pi_{4,n}\rightarrow \Pi_{4,N_L}\times \Pi_{4,N_R}$ %\textcolor{red}{Should this map be $\Pi_{4,n}\to\Pi_{4,N_R}\times\Pi_{4,N_L}$?} 
which is closely related but slightly different from our $\Phi_a$:
$$
\Psi(v_1,\ldots,v_n)\mapsto (v_1,\ldots,v_{a},\widetilde{u_1},v_n),(v_{a+1},\ldots,v_{n-2},\widetilde{u_2},\widetilde{u_3}).
$$
Here $\widetilde{u_i}$ are defined similarly to our $u_i$, but are distributed between both matrices $V_1,V_2$. As a result, the effect of $\Psi$ on the cluster structure is significantly more complicated, see \cite[Theorem 4.7]{ELPTSBW}.
It would be interesting to find a knot-theoretic interpretation of $\Psi$ and   generalizations of $\Psi$ for $k\neq 4$.

The paper is organized as follows. In Section \ref{sec: background} we recall some background on open positroids and cluster structures. In Section \ref{sec: construction} we describe the construction and the main properties of the map $\Phi_a$, and prove Theorem \ref{thm: intro main}(a). In Section \ref{sec: construction} we consider the cluster structures on all varieties in question and prove Theorem \ref{thm: intro main}(b,c). Finally, in Section \ref{sec: braids} we relate our constructions to braid varieties.

\section*{Acknowledgments}

We are grateful to Roger Casals, Thomas Lam, Melissa Sherman-Bennett, Jos\'e Simental and Lauren Williams for inspiring discussions. This work was partially supported by the NSF grant DMS-2302305.

\section{Background}
\label{sec: background}

\subsection{Grassmannians and open positroids}

The Grassmannian $\Gr(k,n)$ is the space of all $k$-dimensional subspaces of $\C^n$, presented as the row span of a $k\times n$ matrix $V$ of maximal rank. Let $v_1,\ldots,v_n$ be the columns of $V$ where $v_i$ are $k$-dimensional vectors. Given an ordered subset $I\in\binom{[n]}{k}$, the \textit{Pl\"ucker coordinate} $\Delta_I(V)$ is the minor of $k\times k$ submatrix of $V$ in column set $I$. We will sometimes consider the exterior algebra $\wedge^{\bullet}\C^k$, and identify $\Delta_I(V)$ with $v_{i_1}\wedge\cdots \wedge v_{i_k}\in \wedge^k(\C^k)\simeq \C$ for $I=\{i_1,\ldots,i_k\}$.

The row operations have the effect of changing $V$ to $AV$ for an invertible $k\times k$ matrix $A$. This implies $v_i\mapsto Av_i$ and $\Delta_I\mapsto \det(A)\Delta_I$ for all $I$. In particular, $\Delta_I$ can be considered as projective coordinates on $\Gr(k,n)$, or as affine coordinates on the affine cone $\widehat{\Gr}(k,n)$.

Knutson-Lam-Speyer \cite{KLS} constructed the stratification $$\Gr(k,n)=\displaystyle\bigsqcup_{f\in\textbf{B}_{k,n}}\Pi_f^{\circ}$$ where $\Pi_f^{\circ}$ are open positroid varieties indexed by a finite set $\textbf{B}_{k,n}$ of bounded  affine permutations, see \cite[Section 4.1]{GL} for more information. This positroid stratification contains a unique open stratum, the top dimensional positroid variety, defined such that cyclically consecutive Pl\"ucker coordinates are non-vanishing, i.e.,
$$
\Pi_{k,n}^{\circ}:=\left\{V\in\Gr(k,n):\Delta_{1,2,\dots,k}(V),\Delta_{2,3,\dots,k+1}(V),\dots,\Delta_{n,1,2,\dots,k-1}(V)\ne0\right\}.$$

%TODO: notations, minors, positroids and all that \textcolor{red}{Done?}

\subsection{Cluster algebras}

Cluster algebras are an interesting class of commutative rings which are defined by a seed $s$ consisting of a quiver, or exchange matrix, and cluster variables, which are a finite collection of algebraically independent elements of the algebra. This seed along with a concept of mutation generates a subring of a field $\mathcal{F}$. For more details on cluster algebras, see \cite{W}.

A cluster variety is an affine algebraic variety $X$ defined by a collection of open charts $U\simeq(\C^*)^d$ where each chart $U$ is parametrized by cluster coordinates $A_1,\ldots,A_d$ which are invertible on $U$ and extend to regular functions on $X$. If the coordinate extends to a non-vanishing regular function on $X$ then we call it frozen, otherwise we call the coordinate  mutable. 

For each chart we assign a skew-symmetric integer matrix $\varepsilon_{ij}$ called the exchange matrix to a quiver $Q$ defined by 
$$
(\varepsilon_{ij})=\begin{cases}
    a &\text{if there are $a$ arrows from vertex $i$ to vertex $j$;}\\
    -a &\text{if there are $a$ arrows from vertex $j$ to vertex $i$;}\\
    0 &\text{otherwise}
\end{cases}$$
For each chart $U$ and each mutable variable $A_k$, there is another chart $U'$ with cluster coordinates $A_1,\ldots,A'_k,\ldots,A_d$ and a skew-symmetric matrix $\varepsilon'_{ij}$ related by mutation $\mu_k$, where the mutation is defined by 
\begin{equation}
\label{eq: mutation}
A_k'A_k=\left(\displaystyle\prod_{\varepsilon_{ki\geq0}}A_i^{\varepsilon_{ki}}+\prod_{\varepsilon_{ki\leq0}}A^{-\varepsilon_{ki}}_i\right)
\end{equation}
If $i\ne k$ then the cluster variables $A_i$ remain unchanged.

When performing a mutation, we modify the quiver using the following rules:
\begin{enumerate}
    \item If there is a path of the vertices $i\rightarrow k\rightarrow j$, then we add an arrow from $i$ to $j$.
    \item Any arrows incident to $k$ change orientation.
    \item Remove a maximal disjoint collection of 2-cycles produced in Steps (1) and (2).
\end{enumerate}

We relate any two charts in the cluster algebra by a sequence of mutations $\underline{\mu}$, and we note that $\mu_k$ is an involution. Given these conditions the ring of functions on $X$ is generated by all cluster variables in all charts.

We define exchange ratio $\widehat{y_i}$ as the ratio of two terms in \eqref{eq: mutation}:
$$
\widehat{y_i}=\frac{\prod_{\varepsilon_{ki\geq0}}A_i^{\varepsilon_{ki}}}{\prod_{\varepsilon_{ki\leq0}}A^{-\varepsilon_{ki}}_i}.
$$
We will need the notion of exchange ratios in the definition of quasi-equivalences.
Let $X$ be a rational affine algebraic variety with algebra of regular functions $\C[X]$ and field of rational functions $\C(X)$. 
\begin{definition}\cite{F, FSB}
\label{def: quasi}
    Let $\Sigma$ and $\Sigma_0$ be seeds of rank $r$ in $\C(X)$. Let $Q, A_i, \hat{y_i}$ denote the  quiver, cluster variables and exchange ratios in $\Sigma$ and use primes to denote these quantities in $\Sigma_0$. We assume that $A_{r+1}, \ldots , A_d$ are frozen.
    Then $\Sigma$ and $\Sigma_0$ are quasi-equivalent, denoted $\Sigma\sim\Sigma_0$, if the following hold:
    \begin{itemize}
        \item The groups $\mathbf{P}, \mathbf{P}_0 \subset \C[X]$ of Laurent monomials in frozen variables coincide. That is, each frozen variable $A'_i$ is a Laurent monomial in $\{A_{r+1}, \ldots, A_d\}$ and vice versa.
        \item Corresponding mutable variables coincide up to multiplication by an element of $\mathbf{P}$: for $i\in[r]$, there is a Laurent monomial $M_i\in\mathbf{P}$ such that $A_i = M_iA'_i\in \C(X)$.
        \item The exchange ratios (3) coincide: $\hat{y}_i = \hat{y}'_i$ for $i\in[r]$.
    \end{itemize}
    Quasi-equivalence is an equivalence relation on seeds.
Seeds $\Sigma,\Sigma_0$ are related by a quasi-cluster transformation if there exists a finite sequence
$\underline{\mu}$ of mutations such that $\underline{\mu}(\Sigma) \sim \Sigma_0$.
\end{definition}

By the main result of \cite{F}, it is sufficient to check the conditions of quasi-equivalence in one cluster, and they will automatically hold in every other cluster.

\subsection{Cluster structures on open positroids}
\label{sec: cluster Scott}

In 2003, Scott \cite{S} established that the homogeneous coordinate ring of $\Gr(k,n)$ denoted $\C[\widehat{\Gr}(k,n)]$ has a cluster structure using Postnikov arrangements. In this paper, will we use a different construction using rectangles seed $\Sigma_{k,n}$ that generate the cluster structure for the Pl\"ucker ring $R_{k,n}$  isomorphic to $\C[\widehat{\Gr}(k,n)]$ as detailed in \cite[Section 6.7]{FWZ}. The cluster structure in the Pl\"ucker ring $R_{k,n}$ is generated from the mutations on the rectangular seed $\Sigma_{k,n}$.  Since (unlike \cite{S}) we always assume that the frozen variables are invertible, we in fact consider a cluster structure on $\C[\widehat{\Pi}_{k,n}^{\circ}]$. 

We first construct the quiver $Q_{k,n}$ where vertices are labeled by rectangles $r$ contained in the $k\times(n-k)$ rectangle $R$ along with the empty rectangle $\varnothing$. The frozen vertices are defined as rectangles of size $k\times j$ for $1\le j\le n-k$, size $i\times(n-k)$ for $1\le i\le k$, and $\emptyset$. The arrows connect from the $i\times j$ rectangle to the $i\times(j+1)$ rectangle, the $(i+1)\times j$ rectangle, and the $(i-1)\times(j-1)$ rectangle with the conditions that the rectangle has nonzero dimension, fits inside of $R$ and does not connect two frozens. There is also an arrow from the $\varnothing$ rectangle to the $1\times1$ rectangle, see Figure \ref{fig: Q_kn}.

Each rectangle $r$ contained in the $k\times(n-k)$ rectangle $R$ corresponds to a $k$-element subset of $[n]$ representing a Pl\"ucker coordinate. This correspondence is determined by positioning $r$ in $R$ such that the upper left corner coincides with the upper left corner of $R$. There exists a path from the upper right corner to the lower left corners of $R$ which traces out the smaller rectangle $r$, with steps from $1$ to $n$, where the map from $r$ to $I(r)$ is given by the vertical steps of the path, see Figure \ref{fig: rect to plucker}. Define $$\tilde{x}^{k,n}=\{\Delta_{I(r)}: r\text{ rectangle contained in }k\times(n-k) \text{ rectangle}\}$$
We may now define the rectangles seed $\Sigma_{k,n}=(\tilde{x}^{k,n},\varepsilon_{ij}(Q_{k,n})).
$

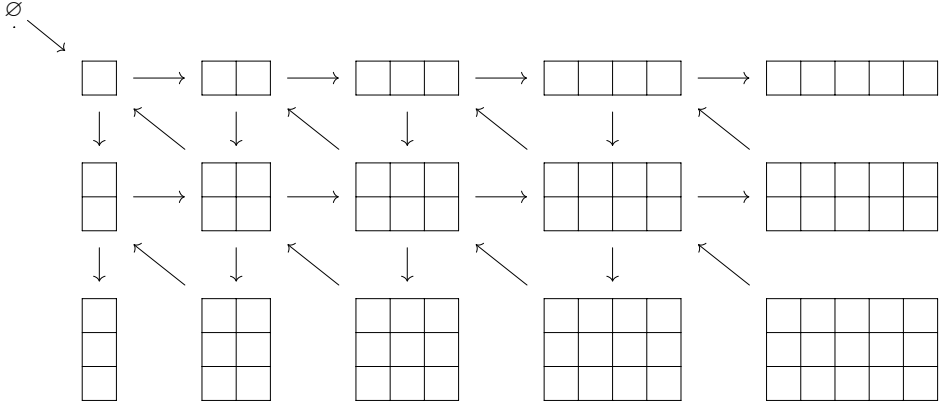
\begin{figure}
    \centering
    \scalebox{.9}{\begin{tikzpicture}[line cap=round,line join=round]
  \pic at (0,0)    {matrix={3}{1}{}{}          {}};
  \pic at (1.75,0)    {matrix={3}{2}{}{}          {}};
  \pic at (4,0)    {matrix={3}{3}{}{}   {}};
  \pic at (6.75,0)    {matrix={3}{4}{}{}     {}};
  \pic at (10,0) {matrix= {3}{5}{}{}{}};
  \pic at (0,2)    {matrix={2}{1}{}{}          {}};
  \pic at (1.75,2)    {matrix={2}{2}{}{}          {}};
  \pic at (4,2)    {matrix={2}{3}{}{}          {}};
  \pic at (6.75,2)    {matrix={2}{4}{}{}          {}};
  \pic at (10,2)    {matrix={2}{5}{}{}          {}};
  \pic at (0,3.5)    {matrix={1}{1}{}{}          {}};
  \pic at (1.75,3.5)    {matrix={1}{2}{}{}          {}};
  \pic at (4,3.5)    {matrix={1}{3}{}{}          {}};
  \pic at (6.75,3.5)    {matrix={1}{4}{}{}          {}};
  \pic at (10,3.5)    {matrix={1}{5}{}{}          {}};
  \pic at (-1,4)    {matrix={0}{0}{$\varnothing$}{}          {}};
\draw[->] (0.75,1.5)--(1.5,1.5);
\draw[->] (0.75,3.25)--(1.5,3.25);
\draw[->] (3,1.5)--(3.75,1.5);
\draw[->] (3,3.25)--(3.75,3.25);
\draw[->] (5.75,1.5)--(6.5,1.5);
\draw[->] (5.75,3.25)--(6.5,3.25);
\draw[->] (9,1.5)--(9.75,1.5);
\draw[->] (9,3.25)--(9.75,3.25);
\draw[->] (0.25,2.75)--(0.25,2.25);
\draw[->] (0.25,0.75)--(0.25,0.25);
\draw[->] (2.25,2.75)--(2.25,2.25);
\draw[->] (2.25,0.75)--(2.25,0.25);
\draw[->] (4.75,2.75)--(4.75,2.25);
\draw[->] (4.75,0.75)--(4.75,0.25);
\draw[->] (7.75,2.75)--(7.75,2.25);
\draw[->] (7.75,0.75)--(7.75,0.25);
\draw[->] (1.5,0.2)--(0.75,0.8);
\draw[->] (3.75,0.2)--(3,0.8);
\draw[->] (6.5,0.2)--(5.75,0.8);
\draw[->] (9.75,0.2)--(9,0.8);

\draw[->] (1.5,2.2)--(0.75,2.8);
\draw[->] (3.75,2.2)--(3,2.8);
\draw[->] (6.5,2.2)--(5.75,2.8);
\draw[->] (9.75,2.2)--(9,2.8);
\draw[->] (-0.8,4.1)--(-0.25,3.65);
\end{tikzpicture}
}
    \caption{The quiver $Q_{3,8}$. Vertices are labeled by rectangles contained in a $3\times 5$ rectangle. The grid is arranged such that the rectangles width increases from left to right and the heights increase from top to bottom.}
    \label{fig: Q_kn}
\end{figure}

%\begin{comment}

\begin{figure}
    \centering
    \scalebox{0.9}{

\tikzset{every picture/.style={line width=0.75pt}} %set default line width to 0.75pt        

\begin{tikzpicture}[x=0.75pt,y=0.75pt,yscale=-1,xscale=1]
%uncomment if require: \path (0,375); %set diagram left start at 0, and has height of 375

%Shape: Rectangle [id:dp8509924983908679] 
\draw  [color={rgb, 255:red, 0; green, 0; blue, 0 }  ,draw opacity=1 ][line width=1.5]  (146.67,70) -- (466.67,70) -- (466.67,250) -- (146.67,250) -- cycle ;
%Shape: Rectangle [id:dp003693418224195666] 
\draw  [color={rgb, 255:red, 74; green, 144; blue, 226 }  ,draw opacity=0.75 ][line width=2.25]  (146.67,70) -- (326.67,70) -- (326.67,170) -- (146.67,170) -- cycle ;
%Shape: Grid [id:dp9339802635540755] 
\draw  [draw opacity=0] (146.67,70) -- (466.67,70) -- (466.67,250) -- (146.67,250) -- cycle ; \draw  [color={rgb, 255:red, 155; green, 155; blue, 155 }  ,draw opacity=0.17 ] (146.67,70) -- (146.67,250)(166.67,70) -- (166.67,250)(186.67,70) -- (186.67,250)(206.67,70) -- (206.67,250)(226.67,70) -- (226.67,250)(246.67,70) -- (246.67,250)(266.67,70) -- (266.67,250)(286.67,70) -- (286.67,250)(306.67,70) -- (306.67,250)(326.67,70) -- (326.67,250)(346.67,70) -- (346.67,250)(366.67,70) -- (366.67,250)(386.67,70) -- (386.67,250)(406.67,70) -- (406.67,250)(426.67,70) -- (426.67,250)(446.67,70) -- (446.67,250) ; \draw  [color={rgb, 255:red, 155; green, 155; blue, 155 }  ,draw opacity=0.17 ] (146.67,70) -- (466.67,70)(146.67,90) -- (466.67,90)(146.67,110) -- (466.67,110)(146.67,130) -- (466.67,130)(146.67,150) -- (466.67,150)(146.67,170) -- (466.67,170)(146.67,190) -- (466.67,190)(146.67,210) -- (466.67,210)(146.67,230) -- (466.67,230)(146.67,250) -- (466.67,250) ; \draw  [color={rgb, 255:red, 155; green, 155; blue, 155 }  ,draw opacity=0.17 ]  ;

% Text Node
\draw (312.67,77) node [anchor=north west][inner sep=0.75pt]  [font=\scriptsize] [align=left] {$\displaystyle a$};
% Text Node
\draw (307.67,116.33) node [anchor=north west][inner sep=0.75pt]   [align=left] {$\displaystyle \vdots $};
% Text Node
\draw (278.67,154) node [anchor=north west][inner sep=0.75pt]  [font=\scriptsize] [align=left] {$\displaystyle a+i-1$};
% Text Node
\draw (298.67,96) node [anchor=north west][inner sep=0.75pt]  [font=\scriptsize] [align=left] {$\displaystyle a+1$};
% Text Node
\draw (78.67,176) node [anchor=north west][inner sep=0.75pt]  [font=\scriptsize] [align=left] {$\displaystyle n-k+i+1$};
% Text Node
\draw (126.33,190.67) node [anchor=north west][inner sep=0.75pt]   [align=left] {$\displaystyle \vdots $};
% Text Node
\draw (115.67,217) node [anchor=north west][inner sep=0.75pt]  [font=\scriptsize] [align=left] {$\displaystyle n-1$};
% Text Node
\draw (128.67,235) node [anchor=north west][inner sep=0.75pt]  [font=\scriptsize] [align=left] {$\displaystyle n$};
% Text Node
\draw (147.67,257.1) node [anchor=north west][inner sep=0.75pt]   [align=left] {$\displaystyle \underbrace{\ \ \ \ \ \ \ \ \ \ \ \ \ \ \ \ \ \ \ \ \ \ \ \ \ \ \ \ \ \ \ \ \ \ \ \ \ \ \ \ \ \ \ \ \ \ \ \ \ \ \ \ \ \ \ \ \ \ \ \ \ \ \ \ \ \ \ \ \ \ \ }_{n-k}$};
% Text Node
\draw (480.08,158.41) node  [rotate=-270] [align=left] {$\displaystyle \underbrace{\ \ \ \ \ \ \ \ \ \ \ \ \ \ \ \ \ \ \ \ \ \ \ \ \ \ \ \ \ \ \ \ \ \ \ \ \ \ \ }$};
% Text Node
\draw (494.17,152) node [anchor=north west][inner sep=0.75pt]  [font=\small] [align=left] {$\displaystyle k$};
% Text Node
\draw (331.67,72.6) node [anchor=north west][inner sep=0.75pt]  [color={rgb, 255:red, 155; green, 155; blue, 155 }  ,opacity=1 ] [align=left] {$\displaystyle \underbrace{\ \ \ \ \ \ \ \ \ \ \ \ \ \ \ \ \ \ \ \ \ \ \ \ \ \ \ \ \ }_{a-1\ steps}$};
% Text Node
\draw (148.67,173) node [anchor=north west][inner sep=0.75pt]  [color={rgb, 255:red, 155; green, 155; blue, 155 }  ,opacity=1 ] [align=left] {$\displaystyle \underbrace{\ \ \ \ \ \ \ \ \ \ \ \ \ \ \ \ \ \ \ \ \ \ \ \ \ \ \ \ \ \ \ \ \ \ \ \ \ \ \ \ }_{n-k-a+1\ steps}$};
% Text Node
\draw (168.67,93) node [anchor=north west][inner sep=0.75pt]  [font=\large,color={rgb, 255:red, 74; green, 144; blue, 226 }  ,opacity=1 ] [align=left] {$\displaystyle r$};
% Text Node
\draw (428.67,213) node [anchor=north west][inner sep=0.75pt]  [font=\large] [align=left] {$\displaystyle R$};

\end{tikzpicture}}
    \caption{The Pl\"ucker coordinate $\Delta_{I(r)}$ corresponding to a rectangle $r$ is given by the vertical steps in the path from the upper right corner to the lower left corner of the rectangle $R$ of size $k\times(n-k)$ that cuts out the rectangle $r$ positioned in the upper left corner of $R$. }
    \label{fig: rect to plucker}
\end{figure}
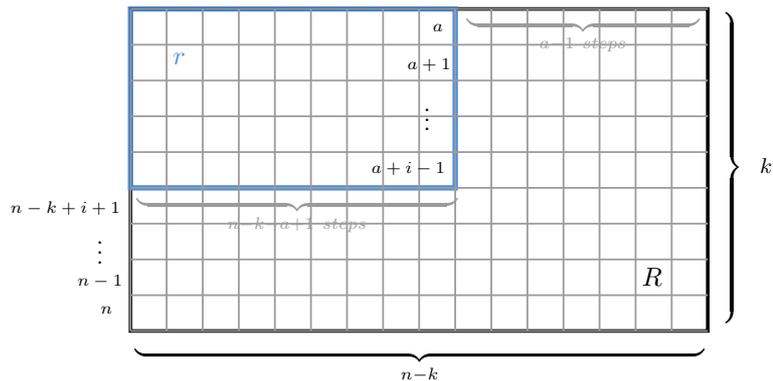
%\end{comment}

\begin{figure}
\begin{tikzcd}
\boxed{\Delta_{678}} \arrow{dr}& & & & & \\
 & \Delta_{578} \arrow{d}\arrow{r}& \Delta_{478} \arrow{d}\arrow{r}& \Delta_{378} \arrow{d}\arrow{r} & \Delta_{278} \arrow{d}\arrow{r} & \boxed{\Delta_{178}}\\
  & \Delta_{568} \arrow{d}\arrow{r}& \Delta_{458} \arrow{d}\arrow{r}\arrow{ul}& \Delta_{348} \arrow{d}\arrow{r}\arrow{ul}& \Delta_{238} \arrow{d}\arrow{r}\arrow{ul}& \boxed{\Delta_{128}}\arrow{ul}\\
   & \boxed{\Delta_{567}} & \boxed{\Delta_{456}} \arrow{ul}& \boxed{\Delta_{345}} \arrow{ul}& \boxed{\Delta_{234}} \arrow{ul}& \boxed{\Delta_{123}}\arrow{ul}\\
\end{tikzcd}
\label{fig: plucker quiver}
\caption{Cluster variables in $\Gr(3,8)$ corresponding to the rectangles seed.}
\end{figure}
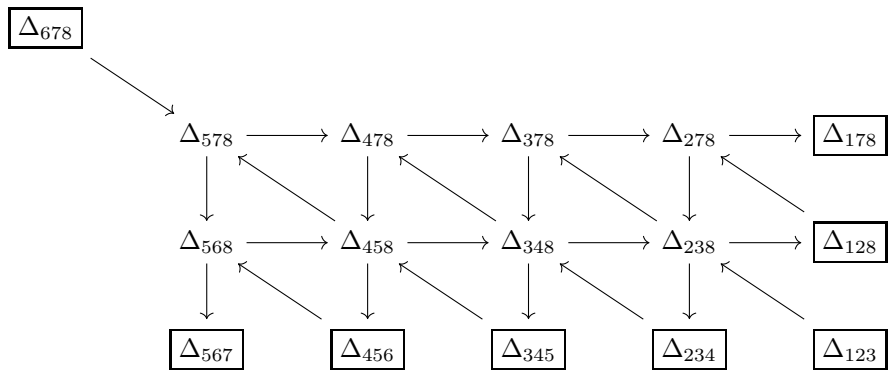

We can summarize (and slightly rephrase) the above constructions as follows. We define ordered subsets 
\begin{equation}
I(a,i)=\{a,a+1,\ldots,a+i-1,n-k+i+1,\ldots,n\},
\end{equation}
where $a=n-k-j+1$.

\begin{theorem}[\cite{S}]
\label{thm: scott}
The cluster variables in the initial seed are given by the minors $\Delta_{I(a,i)}$ for $1\le a\le n-k$ and $1\le i\le k$, and an additional frozen variable $\Delta_{n-k+1,\ldots,n}$. Furthermore:

1) The variables $\Delta_{I(a,i)}$ are frozen for $a=1$ and $i=k$, and mutable otherwise.

2) The quiver $Q_{k,n}$ consists of the following arrows:
\begin{equation}
\label{eq: scott}
\begin{tikzcd}
\Delta_{I(a,i)} \arrow{r}\arrow{d} & \Delta_{I(a-1,i)}\arrow{d}\\
\Delta_{I(a,i+1)} \arrow{r}& \Delta_{I(a-1,i+1)}\arrow{ul}
\end{tikzcd}
\end{equation}
3) There is an additional arrow $\Delta_{n-k+1,\ldots,n}\to \Delta_{I(n-k,1)}$.
\end{theorem}

See Figure \ref{fig: plucker quiver} for the $\Pi_{3,8}^{\circ}$ example.

\section{Construction}
\label{sec: construction}

Let $\Pi_{k,n}^{\circ}$ be the open positroid variety in the Grassmannian $\Gr(k,n)$. 

\begin{definition}
Given $2\le a\le n-k$, we define an open subset $U_a\subset \Pi_{k,n}^{\circ}$ by the inequalities
\begin{equation}
U_a=\{\Delta_{I(a,i)}\neq 0,\ 1\le i\le k-1\}.
\end{equation}

%requiring that the minors
%$$
%\Delta_{a,n-k+2,\ldots,n},\Delta_{a,a+1,n-k+3,\ldots,n},\ldots, \Delta_{a,a+1,\ldots,a+k-2,n}.
%$$
%are all nonzero.
\end{definition}

For $0\le s\le i-1$ and $0\le t\le k-i-1$ we define ordered subsets 
%$$
%I(a,i)=\{a,a+1,\ldots,a+i-1,n-k+i+1,\ldots,n\},
%$$
$$
I'(a,s,i)=\{a,\ldots,a+s-1,a+i,a+s+1,\ldots,a+i-1,n-k+i+1,\ldots,n\},
$$
$$
I'_{\sort}(a,s,i)=\{a,\ldots,a+s-1,a+s+1,\ldots,a+i-1,a+i,n-k+i+1,\ldots,n\},
$$
and
$$
I''(a,t,i)=\{a,\ldots,a+i-1,n-k+i+1,\ldots,n-t-1,a+i,n-t+1,\ldots,n\},
$$
$$
I''_{\sort}(a,t,i)=\{a,\ldots,a+i-1,a+i,n-k+i+1,\ldots,n-t-1,n-t+1,\ldots,n\},
$$
Note that $I'(a,s,i)$ is obtained from $I(a,i)$ by replacing $a+s$ by $a+i$ (without changing the order), while $I''(a,t,i)$ is obtained from $I(a,i)$ by replacing $n-t$ by $a+i$. Also, $I'(a,s,i)$ and $I'_{\sort}(a,s,i)$ are related by an $(i-s)$-cycle while $I''(a,t,i)$ and $I''_{\sort}(a,t,i)$ are related by a $(k-i-1-t)$-cycle.

\begin{lemma}
\label{lem: Cramer}
Given a matrix $V\in U_a$, for all $1\le i\le k-2$ we have an identity
\begin{align*}
v_{a+i}&=\sum_{s=0}^{i-1}\frac{\Delta_{I'(a,s,i)}}{\Delta_{I(a,i)}}v_{a+s}+\sum_{t=0}^{k-i-1}\frac{\Delta_{I''(a,t,i)}}{\Delta_{I(a,i)}}v_{n-t}\\
&=\sum_{s=0}^{i-1}(-1)^{i-s-1}\frac{\Delta_{I'_{\sort}(a,s,i)}}{\Delta_{I(a,i)}}v_{a+s}+\sum_{t=0}^{k-i-1}(-1)^{k-i-t}\frac{\Delta_{I''_{\sort}(a,t,i)}}{\Delta_{I(a,i)}}v_{n-t}
\end{align*}
\end{lemma}

\begin{proof}
Since $\Delta_{I(a,i)}(V)\neq 0$, the vectors 
$
v_{a},\ldots,v_{a+i-1},v_{n-k+i+1},\ldots,v_n
$
span $\C^k$. We can uniquely write $v_{a+i}$ as linear combination of these:
$$
v_{a+i}=x_0v_{a}+\ldots+x_{i-1}v_{a+i-1}+y_{k-i-1}v_{n-k+i+1}+\ldots+y_0v_{n}.
$$
Now the coefficients $x_s,y_t$ are determined by the Cramer's Rule:
$$
x_s=\frac{\Delta_{I'(a,s,i)}}{\Delta_{I(a,i)}}=(-1)^{i-s-1}\frac{\Delta_{I'_{\sort}(a,s,i)}}{\Delta_{I(a,i)}},\ 
y_t=\frac{\Delta_{I''(a,t,i)}}{\Delta_{I(a,i)}}=(-1)^{k-i-t}\frac{\Delta_{I''_{\sort}(a,t,i)}}{\Delta_{I(a,i)}}.
$$
\end{proof}

\begin{example}
\label{ex: 3,8}
    For $a=3$ the open subset $U_3\subset\Pi_{3,8}^{\circ}$ is defined by $\Delta_{348},\Delta_{378}\ne0$, indicating that the vectors $v_3,v_7,v_8$ span $\C^3$ and the vectors $v_3,v_4,v_8$ span $\C^3$. Using Cramer's rule we may express the vector $v_4$ as 
$$
    v_4=\frac{\Delta_{478}}{\Delta_{378}}v_3+\frac{\Delta_{348}}{\Delta_{378}}v_7+\frac{\Delta_{374}}{\Delta_{378}}v_8
    =\frac{\Delta_{478}}{\Delta_{378}}v_3+\frac{\Delta_{348}}{\Delta_{378}}v_7-\frac{\Delta_{347}}{\Delta_{378}}v_8
$$
\end{example}

\begin{example}
\label{ex: 5,10}
    The open subset $U_3\subset\Pi_{5,10}^{\circ}$ is defined by $\Delta_{3,7,8,9,10},\Delta_{3,4,8,9,10},\Delta_{3,4,5,9,10},\Delta_{3,4,5,6,10}\ne0$. From this collection of nonvanishing Pl\"ucker coordinates we have 
$$
\C^5=\langle v_3,v_7,v_8,v_9,v_{10}\rangle=
\langle v_3,v_4,v_8,v_9,v_{10}\rangle=
\langle v_3,v_4,v_5,v_9,v_{10}\rangle.
$$
%that the appropriate combination of vectors $v_3,v_4,v_5,v_6,v_7,v_8,v_9,v_{10}$ span $\C^5$. Given that $k=5$, we may express the $k-2=3$ vectors $v_4,v_5,v_6$ in terms of the spanning vectors above using Cramer's rule. 
By Cramer's Rule, we can expand $v_4,v_5,v_6$ in the respective bases:
\begin{align*}
    v_4%&=\frac{\Delta_{4,7,8,9,10}}{\Delta_{3,7,8,9,10}}v_3+\frac{\Delta_{3,4,8,9,10}}{\Delta_{3,7,8,9,10}}v_7+\frac{\Delta_{3,7,4,9,10}}{\Delta_{3,7,8,9,10}}v_8+\frac{\Delta_{3,7,8,4,10}}{\Delta_{3,7,8,9,10}}v_9+\frac{\Delta_{3,7,8,9,4}}{\Delta_{3,7,8,9,10}}v_{10}\\
    &=\frac{\Delta_{4,7,8,9,10}}{\Delta_{3,7,8,9,10}}v_3+\frac{\Delta_{3,4,8,9,10}}{\Delta_{3,7,8,9,10}}v_7-\frac{\Delta_{3,4,7,9,10}}{\Delta_{3,7,8,9,10}}v_8+\frac{\Delta_{3,4,7,8,10}}{\Delta_{3,7,8,9,10}}v_9-\frac{\Delta_{3,4,7,8,9}}{\Delta_{3,7,8,9,10}}v_{10}\\
    \vspace{.25cm}
    v_5
    %&=\frac{\Delta_{5,4,8,9,10}}{\Delta_{3,4,8,9,10}}v_3+\frac{\Delta_{3,5,8,9,10}}{\Delta_{3,4,8,9,10}}v_4+\frac{\Delta_{3,4,5,9,10}}{\Delta_{3,4,8,9,10}}v_8+\frac{\Delta_{3,4,8,5,10}}{\Delta_{3,4,8,9,10}}v_9+\frac{\Delta_{3,4,8,9,5}}{\Delta_{3,4,8,9,10}}v_{10}\\
    &=-\frac{\Delta_{4,5,8,9,10}}{\Delta_{3,4,8,9,10}}v_3+\frac{\Delta_{3,5,8,9,10}}{\Delta_{3,4,8,9,10}}v_4+\frac{\Delta_{3,4,5,9,10}}{\Delta_{3,4,8,9,10}}v_8-\frac{\Delta_{3,4,5,8,10}}{\Delta_{3,4,8,9,10}}v_9+\frac{\Delta_{3,4,5,8,9}}{\Delta_{3,4,8,9,10}}v_{10}\\
    \vspace{.25cm}
    v_6
    %&=\frac{\Delta_{6,4,5,9,10}}{\Delta_{3,4,5,9,10}}v_3+\frac{\Delta_{3,6,5,9,10}}{\Delta_{3,4,5,9,10}}v_4+\frac{\Delta_{3,4,6,9,10}}{\Delta_{3,4,5,9,10}}v_5+\frac{\Delta_{3,4,5,6,10}}{\Delta_{3,4,5,9,10}}v_9+\frac{\Delta_{3,4,5,9,6}}{\Delta_{3,4,5,9,10}}v_{10}\\
    &=\frac{\Delta_{4,5,6,9,10}}{\Delta_{3,4,5,9,10}}v_3-\frac{\Delta_{3,5,6,9,10}}{\Delta_{3,4,5,9,10}}v_4+\frac{\Delta_{3,4,6,9,10}}{\Delta_{3,4,5,9,10}}v_5+\frac{\Delta_{3,4,5,6,10}}{\Delta_{3,4,5,9,10}}v_9-\frac{\Delta_{3,4,5,6,9}}{\Delta_{3,4,5,9,10}}v_{10}
\end{align*}
\end{example}

\begin{definition}
Given a matrix $V=(v_1,\ldots,v_n)$ in $U_a$, we define two matrices $V_1,V_2$ as follows:
\begin{equation}
\label{eq: V1 V2}
V_1=(v_{a},\ldots,v_n),\ V_2=(v_1,\ldots,v_a,u_1\ldots,u_{k-2},v_n)
\end{equation}
where 
\begin{equation}
\label{eq: def u}
u_i=v_{a+i}-\sum_{s=0}^{i-1}\frac{\Delta_{I'(a,s,i)}}{\Delta_{I(a,i)}}v_{a+s}=\sum_{t=0}^{k-i-1}\frac{\Delta_{I''(a,t,i)}}{\Delta_{I(a,i)}}v_{n-t}.
\end{equation}
\end{definition}
The second equation in \eqref{eq: def u} follows from Lemma \ref{lem: Cramer}.

\begin{example}
\label{ex: 3 8 V1V2}
    Continuing Example \ref{ex: 3,8}, we decompose the matrix 
    $$
    V=\begin{pmatrix}
        v_1 &v_2 &v_3 &v_4 &v_5 &v_6 &v_7 &v_8
    \end{pmatrix}$$
    into 
    $$V_1=\begin{pmatrix}
        v_3 &v_4 &v_5 &v_6 &v_7 &v_8
    \end{pmatrix},\quad  V_2=\begin{pmatrix}
        v_1 &v_2 &v_3 &u &v_8
    \end{pmatrix}$$
    where $$u=v_4-\frac{\Delta_{478}}{\Delta_{378}}v_3=\frac{\Delta_{348}}{\Delta_{378}}v_7-\frac{\Delta_{347}}{\Delta_{378}}v_8.$$
\end{example}

\begin{lemma}
\label{lem: intersection}
Assume that $V\in U_a$. Then for $1\le i\le k-2$ the intersection of two subspaces
$$
\langle v_a,v_{a+1},\ldots,v_{a+i}\rangle\cap 
\langle v_{n-k+i+1},v_{n-k+i+2},\ldots,v_{n}\rangle.
$$
is one-dimensional and spanned by the vector $u_i$.
\end{lemma}

\begin{proof}
By \eqref{eq: def u} and Lemma \ref{lem: Cramer} the vector $u_i$ is indeed contained in this intersection. Since $\Delta_{I(a,i+1)}\neq 0$, the vectors $v_a,\ldots,v_{a+i}$ are linearly independent and hence $u_i\neq 0$.
Since $\Delta_{I(a,i)}\neq 0$, the vectors $v_{n-k+i+1},\ldots,v_n$ are linearly independent as well and altogether the two subspaces span $\C^k$. Now
$$
\dim \langle v_a,v_{a+1},\ldots,v_{a+i}\rangle\cap 
\langle v_{n-k+i+1},v_{n-k+i+2},\ldots,v_{n}\rangle=(i+1)+(k-i)-k=1.
$$
\end{proof}

\begin{lemma}
\label{lem: u wedge}
a) We have 
$$
v_a\wedge u_1\wedge\cdots \wedge u_i=v_a\wedge v_{a+1}\wedge\cdots \wedge v_{a+i}.
$$
b)
We have 
$$
u_i\wedge\cdots \wedge u_{k-2}\wedge v_n= \frac{\Delta_{I(a,k-1)}}{\Delta_{I(a,i)}}v_{n-k+i+1}\wedge\cdots \wedge v_{n-1}\wedge v_{n}.
$$
\end{lemma}

\begin{proof}
a) By \eqref{eq: def u} we have 
$$
u_{i}\in v_{a+i}+\langle v_a,\ldots,v_{a+i-1}\rangle, 
$$
so
$$
v_a\wedge u_1\wedge\cdots \wedge u_i=
v_a\wedge (v_{a+1}+\ldots)\wedge \cdots (v_{a+i}+\ldots)=v_a\wedge v_{a+1}\wedge\cdots \wedge v_{a+i}.
$$

b) Similarly, by the second equation in \eqref{eq: def u} we have 
$$
u_{i}\in \frac{\Delta_{I''(a,k-i-1,i)}}{\Delta_{I(a,i)}}v_{n-k+i+1}+\langle v_{n-k+i+2},\ldots,v_n\rangle. 
$$
Note that $I''(a,k-i-1,i)$ is obtained from $I(a,i)=\{a,\ldots,a-i-1,n-k+i+1,\ldots,n\}$ by replacing $n-k+i+1$ with $a+i$, so in fact 
$I''(a,k-i-1,i)=I(a,i+1)$. Now
$$
u_i\wedge\cdots \wedge u_{k-2}\wedge v_n=
\left(\frac{\Delta_{I(a,i+1)}}{\Delta_{I(a,i)}}v_{n-k+i+1}+\ldots\right)\wedge\cdots \wedge 
\left(\frac{\Delta_{I(a,k-1)}}{\Delta_{I(a,k-2)}}v_{n-1}+\ldots\right)\wedge v_n=
$$
$$
\frac{\Delta_{I(a,i+1)}}{\Delta_{I(a,i)}}\cdot \frac{\Delta_{I(a,i+2)}}{\Delta_{I(a,i+1)}}\cdots \frac{\Delta_{I(a,k-1)}}{\Delta_{I(a,k-2)}}v_{n-k+i+1}\wedge\cdots \wedge v_n.
$$
The factors in the coefficient cancel pairwise except for 
$\Delta_{I(a,k-1)}/\Delta_{I(a,i)}.$
\end{proof}

\begin{lemma}
\label{lem: V1 V2 positroids}
If $V\in U_a$ then $V_1\in \Pi_{k,n-a+1}^{\circ}$ and $V_2\in \Pi_{k,a+k-1}^{\circ}$.
\end{lemma}

\begin{proof}
The first statement is clear by the definition of $U_a$. To prove the second one, we need to compute the following minors: 

1) $\Delta_{b,\ldots,b+k-1}(V_2)$, $b+k-1\le a$. This minor does not change, so $\Delta_{b,\ldots,b+k-1}(V_2)=\Delta_{b,\ldots,b+k-1}(V)\neq 0$.

2) $\Delta_{b,\ldots,b+k-1}(V_2)$, $b<a<b+k-1$. Let $i=b+k-1-a$, then by Lemma \ref{lem: u wedge}(a) we get
$$
\Delta_{b,\ldots,b+k-1}(V_2)=v_b\wedge \cdots \wedge v_a\wedge u_1\wedge \cdots \wedge u_i=$$
$$v_b\wedge \cdots \wedge v_a\wedge v_{a+1}\wedge \cdots \wedge v_{a+i}=\Delta_{b,\ldots,b+k-1}(V)\neq 0.
$$

3) $\Delta_{a,\ldots,a+k-1}(V_2)=v_a\wedge u_1\wedge \cdots \wedge u_{k-2}\wedge v_n=v_a\wedge \cdots \wedge v_{a+k-2}\wedge v_n\neq 0$ by definition of $U_a$.

4) Finally, we need to consider the minor $u_{i}\wedge\cdots \wedge u_{k-2}\wedge v_n\wedge v_1\cdots v_{i}$ which by Lemma \ref{lem: u wedge}(b) equals 
$$
\frac{\Delta_{I(a,k-1)}}{\Delta_{I(a,i)}}v_{n-k+i+1}\wedge \cdots \wedge v_n\wedge v_1\cdots v_{i}=\frac{\Delta_{I(a,k-1)}}{\Delta_{I(a,i)}}\Delta_{n-k+i+1,\ldots,n,1,\ldots,i}\neq 0.
$$

\end{proof}

\begin{theorem}
The map $\Phi_a: V\mapsto (V_1,V_2)$ defined by \eqref{eq: V1 V2} is an isomorphism  between $U_a\subset \Pi_{k,n}^{\circ}$ and the product $\Pi_{k,n-a+1}^{\circ}\times \Pi_{k,a+k-1}^{\circ}$.
\end{theorem}

\begin{proof}
By Lemma \ref{lem: V1 V2 positroids} the map $\Phi_a:U_a\to Pi_{k,n-a+1}^{\circ}\times \Pi_{k,a+k-1}^{\circ}$ is well defined.
We need to construct the inverse map, reconstructing $V$ from $V_1$ and $V_2$. Since $V_1$ and $V_2$ are both defined up to row operations, we need to choose appropriate representatives in their equivalence classes and make sure that they glue correctly to $V$.

For $V_1$, choose a representative in the equivalence class arbitrarily and label the column vectors by $(v_a,\ldots,v_n)$. Since $V_1\in \Pi_{k,n-a+1}^{\circ}$, we have $\Delta_{I(a,i)}\neq 0$.
By Lemma \ref{lem: Cramer}, we can define the vectors $u_1,\ldots,u_{k-2}$ by \eqref{eq: def u}. Applying row operations to $V_1$ is equivalent to the multiplication by an invertible $(k\times k)$ matrix $A$ on the left. It transforms $v_i$ to $Av_i$, multiplies all the minors of $V_1$ by $\det A$, and transforms $u_i$ to
\begin{equation}
\label{eq: u row operations}
u_i\to Av_{a+i}-\sum_{j=0}^{i-1}\frac{\Delta_{I'(a,j,i)}\det(A)}{\Delta_{I(a,i)}\det(A)}(Av_{a+j})=A\left[v_{a+i}-\sum_{j=0}^{i-1}\frac{\Delta_{I'(a,j,i)}}{\Delta_{I(a,i)}}v_{a+j}\right]=Au_i.
\end{equation}

By Lemma \ref{lem: u wedge}(a) we get 
$v_a\wedge u_1\cdots u_{k-2}\wedge v_{n}=v_{a}\wedge v_{a+1}\cdots v_{a+k-2} \wedge v_n$. This is nonzero since $V_1\in \Pi_{k,n-a+1}^{\circ}$, so the vectors $v_a,u_1,\ldots,u_{k-2},v_n$ form a basis of $\C^k$. Therefore we can uniquely find a representative for $V_2$ 
of the form $V_2=(v_1,\ldots,v_{a-1},v_a,u_1,\ldots,u_{k-2},v_n)$. Indeed, if $V'_2=(v_1',\ldots,v'_{a+k-1})$ is some other representative then
$$
V_2=(v_a,u_1,\ldots,u_{k-2},v_n)(v'_{a},\ldots,v'_{a+k-1})^{-1}V'_2.
$$
By \eqref{eq: u row operations}, row operations $V_1\mapsto AV_1$ also change $V_2\to AV_2$. Now we can define
$V=(v_1,\ldots,v_{a-1},v_a,\ldots,v_n)$ where the vectors $v_1,\ldots,v_{a-1}$ are the first $(a-1)$ columns of $V_2$ and $(v_a,\ldots,v_n)=V_1$. By the above, this is well defined up to row operations.

Similarly to the proof of Lemma \ref{lem: V1 V2 positroids}, one can check that $V\in \Pi_{k,n}^{\circ}$, and $V_1\in \Pi_{k,n-a+1}^{\circ}$ immediately implies that $V\in U_a$. This completes the proof.
\end{proof}

\section{Cluster algebra interpretation}
\label{sec: cluster}

 We would like to compare the quivers and cluster coordinates \eqref{eq: scott} for the matrices $V$, $V_1$ and $V_2$, which we denote by $Q_V,Q_{V_1}$ and $Q_{V_2}$.
 By construction, the empty rectangle in both $Q_V$ and $Q_{V_1}$ corresponds to $\Delta_{n-k+1,\ldots,n}(V)$. On the other hand, by Lemma \ref{lem: u wedge}(a) the empty rectangle in $Q_{V_2}$ corresponds to the minor
 $$
 \Delta_{I(a,k-1)}(V_2)=v_a\wedge u_1\wedge \cdots u_{k-2}\wedge v_n=v_{a}\wedge v_{a+1}\wedge \cdots v_{a+k-2}\wedge v_n=\Delta_{I(a,k-1)}(V)=\Delta_{I(1,k-1)}(V_1)
 $$
 which is connected to $\Delta_{a-1,1}(V_2)$.
 
Clearly, the open subset $U_a\subset \Pi^{\circ}_{k,n}$ is defined by freezing the cluster variables $\Delta_{I(a,i)}(V)$ in $Q_V$, which are identified with $\Delta_{I(1,i)}(V_1)$. 
We need to analyze the behavior of all other minors in $Q_V$ under $\Phi_a$.

\begin{lemma}
\label{lem: new minors}
a) If $b\ge a$ then $\Delta_{I(b,i)}(V)=\Delta_{I(b-a+1,i)}(V_1)$.

b) If $b<a$ then 
$$\frac{\Delta_{I(a,k-1)}}{\Delta_{I(a,i)}}\Delta_{I(b,i)}(V)=\Delta_{I(b,i)}(V_2).$$
\end{lemma}

\begin{proof}
Part (a) is clear from \eqref{eq: V1 V2}. For part (b), we first assume $b+i-1\ge a$ and write 
$$
\Delta_{I(b,i)}(V_2)=v_{b}\wedge \cdots v_{a}\wedge (u_1\wedge \cdots \wedge u_{i-(a-b+1)})\wedge (u_{i}\wedge \cdots \wedge u_{k-2}\wedge v_n).
$$
By Lemma \ref{lem: u wedge} we get
$$
v_a\wedge u_1\wedge \cdots \wedge u_{i-(a-b+1)}=v_a\wedge v_{a+1}\wedge \cdots \wedge v_{b+i-1}
$$
and
$$
u_{i}\wedge \cdots \wedge u_{k-2}\wedge v_n=\frac{\Delta_{I(a,k-1)}}{\Delta_{I(a,i)}}v_{n-k+i+1}\wedge\cdots \wedge v_{n-1}\wedge v_{n},
$$
so
$$
\Delta_{I(b,i)}(V_2)=\frac{\Delta_{I(a,k-1)}}{\Delta_{I(a,i)}}(v_b\wedge \cdots v_{b+i-1})\wedge
(v_{n-k+i+1}\wedge\cdots \wedge v_{n-1}\wedge v_{n})=
\frac{\Delta_{I(a,k-1)}}{\Delta_{I(a,i)}}\Delta_{I(b,i)}(V)
$$
Similarly, if $b+i-1<a$ then 
$$
\Delta_{I(b,i)}(V_2)=v_{b}\wedge \cdots v_{b+i-1}\wedge (u_{i}\wedge \cdots \wedge u_{k-2}\wedge v_n)=
$$
$$
\frac{\Delta_{I(a,k-1)}}{\Delta_{I(a,i)}}(v_b\wedge \cdots v_{b+i-1})\wedge
(v_{n-k+i+1}\wedge\cdots \wedge v_{n-1}\wedge v_{n})=
\frac{\Delta_{I(a,k-1)}}{\Delta_{I(a,i)}}\Delta_{I(b,i)}(V)
$$
\end{proof}

%For $k=3,n=8$, we analyze all possible options for freezing a column.

%\begin{example}
%\label{ex: 3 8 a=1}

%\end{example}

%\begin{example}
%\label{ex: 3 8 a=2}

%\end{example}

\begin{example}
\label{ex: 3 8 a=3}
For $a=3$, $k=3,n=8$ we get $\Delta_{378},\Delta_{348}\neq 0$, as in Example \ref{ex: 3 8 V1V2}. We have $V_1=(v_3,v_4,v_5,v_6,v_7,v_8)$ and $V_2=(v_1,v_2,v_3,u,v_8)$. The quiver $Q_V$ after freezing $\Delta_{378}$ and $\Delta_{348}$ has the form:

$$
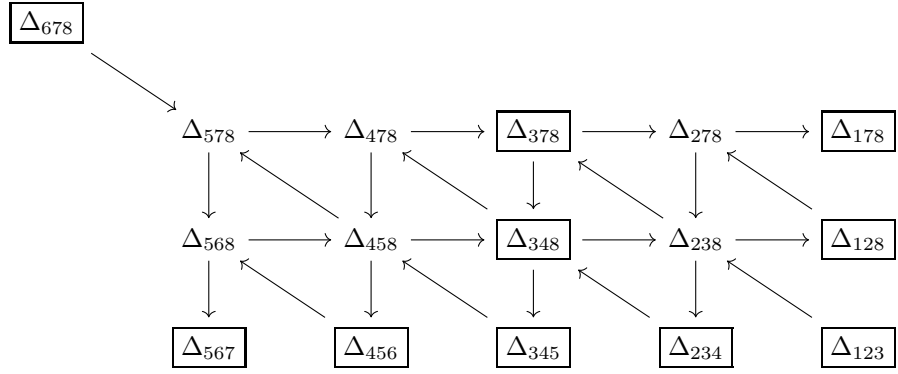
\begin{tikzcd}
\boxed{\Delta_{678}} \arrow{dr}& & & & & \\
 & \Delta_{578} \arrow{d}\arrow{r}& \Delta_{478} \arrow{d}\arrow{r}& \boxed{\Delta_{378}} \arrow{d}\arrow{r} & \Delta_{278} \arrow{d}\arrow{r} & \boxed{\Delta_{178}}\\
  & \Delta_{568} \arrow{d}\arrow{r}& \Delta_{458} \arrow{d}\arrow{r}\arrow{ul}& \boxed{\Delta_{348}} \arrow{d}\arrow{r}\arrow{ul}& \Delta_{238} \arrow{d}\arrow{r}\arrow{ul}& \boxed{\Delta_{128}}\arrow{ul}\\
   & \boxed{\Delta_{567}} & \boxed{\Delta_{456}} \arrow{ul}& \boxed{\Delta_{345}} \arrow{ul}& \boxed{\Delta_{234}} \arrow{ul}& \boxed{\Delta_{123}}\arrow{ul}\\
\end{tikzcd}
$$
while the quivers $Q_{V_1}$ and $Q_{V_2}$ have the form

$$
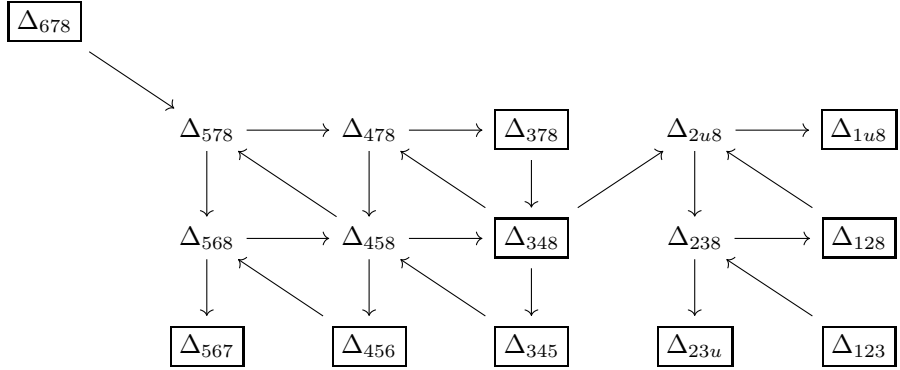
\begin{tikzcd}
\boxed{\Delta_{678}} \arrow{dr}& & & & & \\
 & \Delta_{578} \arrow{d}\arrow{r}& \Delta_{478} \arrow{d}\arrow{r}& \boxed{\Delta_{378}} \arrow{d} & \Delta_{2u8} \arrow{d}\arrow{r} & \boxed{\Delta_{1u8}}\\
  & \Delta_{568} \arrow{d}\arrow{r}& \Delta_{458} \arrow{d}\arrow{r}\arrow{ul}& \boxed{\Delta_{348}} \arrow{d}\arrow{ul} \arrow{ur}& \Delta_{238} \arrow{d}\arrow{r}& \boxed{\Delta_{128}}\arrow{ul}\\
   & \boxed{\Delta_{567}} & \boxed{\Delta_{456}} \arrow{ul}& \boxed{\Delta_{345}} \arrow{ul}& \boxed{\Delta_{23u}} & \boxed{\Delta_{123}}\arrow{ul}\\
\end{tikzcd}
$$
Note that we identified $
\Delta_{3u8}=\Delta_{348}$.
We claim that the two cluster structures are related by a  quasi-equivalence. Indeed, 
$$
\Delta_{3u8}=\Delta_{348},\
\Delta_{23u}=\Delta_{234},\
\Delta_{2u8}=\alpha\Delta_{278},\ \Delta_{1u8}=\alpha\Delta_{178}\ \mathrm{where}\  \alpha=\frac{\Delta_{348}}{\Delta_{378}},
$$
and all other cluster variables are unchanged. Therefore all cluster variables are the same up to monomials in frozen. We need to check the exchange ratios:
$$
y_{2u8}(V_2)=\frac{\Delta_{348}\Delta_{128}}{\Delta_{1u8}\Delta_{238}}=\alpha^{-1}\frac{\Delta_{128}\Delta_{348}}{\Delta_{178}\Delta_{238}}=\frac{\Delta_{128}\Delta_{378}}{\Delta_{178}\Delta_{238}}=y_{278}(V).
$$
while
$$
y_{238}(V_2)=\frac{\Delta_{2u8}\Delta_{123}}{\Delta_{128}\Delta_{23u}}=\alpha\frac{\Delta_{278}\Delta_{123}}{\Delta_{128}\Delta_{234}}=\frac{\Delta_{278}\Delta_{123}\Delta_{348}}{\Delta_{128}\Delta_{234}\Delta_{378}}=y_{238}(V).
$$
Since the exchange ratios agree, we indeed get a quasi-equivalence.

 \end{example}

We are ready to state and prove our main result.

\begin{theorem}
The map $\Phi_a: V\mapsto (V_1,V_2)$ defined by \eqref{eq: V1 V2} is a cluster quasi-isomorphism between $\widehat{U_a}\subset \widehat{\Pi}_{k,n}^{\circ}$ and the product $\widehat{\Pi}_{k,n-a+1}^{\circ}\times \Pi_{k,a+k-1}^{\circ}$.
\end{theorem}

\begin{proof} 
By Lemma \ref{lem: new minors}(a) all Scott minors $\Delta_{I(b,i)}(V_1)$ are the same as the minors in the left half of $Q_V$. 

We need to analyze the right half of $Q_V$.
By Lemma \ref{lem: new minors}(b) all minors in the right half are multiplied by some monomials in $\Delta_{I(a,i)}$ which are frozen on $U_a$. It remains to compute the exchange ratios. We have the following cases:

(a) Interior: $b<a$, $i>1$. The piece of the quiver $Q_V$ around $\Delta_{I(b,i)}$ has the form
$$
\begin{tikzcd}
\Delta_{I(b+1,i-1)} \arrow{r}\arrow{d} & \Delta_{I(b,i-1)} \arrow{r}  \arrow{d} & \Delta_{I(b-1,i-1)} \arrow{d}\\
\Delta_{I(b+1,i)} \arrow{r}\arrow{d}\ & \Delta_{I(b,i)} \arrow{d}\arrow{r}\arrow{ul} & \Delta_{I(b-1,i)}\arrow{d} \arrow{ul}\\
\Delta_{I(b+1,i+1)}\arrow{r} & \Delta_{I(b,i+1)}\arrow{r} \arrow{ul}& \Delta_{I(b-1,i+1)}\arrow{ul}
\end{tikzcd}
$$
The exchange ratios are equal to
$$
y_{I(b,i)}=\frac{\Delta_{I(b,i-1)}\Delta_{I(b+1,i)}\Delta_{I(b-1,i+1)}}{\Delta_{I(b-1,i)}\Delta_{I(b,i+1)}\Delta_{I(b+1,i-1)}}
$$
so by Lemma \ref{lem: new minors} we get
$$
\frac{y_{I(b,i)}(V)}{y_{I(b,i)}(V_2)}=\frac{\Delta_{I(a,i-1)}\Delta_{I(a,i)}\Delta_{I(a,i+1)}}{\Delta_{I(a,i)}\Delta_{I(a,i+1)}\Delta_{I(a,i-1)}}=1.
$$
and $y_{I(b,i)}(V)=y_{I(b,i)}(V_2)$. Note that $\Delta_{I(a,k-1)}$ cancels out.

(b) Top boundary $i=1$:
$$
\begin{tikzcd}
\Delta_{I(b+1,1)} \arrow{r}\arrow{d} & \Delta_{I(b,1)} \arrow{r}  \arrow{d} & \Delta_{I(b-1,1)} \arrow{d}\\
\Delta_{I(b+1,2)} \arrow{r}\ & \Delta_{I(b,2)} \arrow{r}\arrow{ul} & \Delta_{I(b-1,2)} \arrow{ul}
\end{tikzcd}
$$
The exchange ratios are equal to
$$
y_{I(b,1)}=\frac{\Delta_{I(b+1,1)}\Delta_{I(b-1,2)}}{\Delta_{I(b-1,1)}\Delta_{I(b,2)}}
$$
so by Lemma \ref{lem: new minors} we get
$$
\frac{y_{I(b,1)}(V)}{y_{I(b,1)}(V_2)}=\frac{\Delta_{I(a,1)}\Delta_{I(a,2)}}{\Delta_{I(a,1)}\Delta_{I(a,2)}}=1.
$$
and $y_{I(b,i)}(V)=y_{I(b,i)}(V_2)$. Note that $\Delta_{I(a,k-1)}$ cancels again.

(c) Left boundary, $b=a-1$:
$$
\begin{tikzcd}
\Delta_{I(a-1,i-1)} \arrow{r}\arrow{d} & \Delta_{I(a-2,i-1)}    \arrow{d}  \\
\Delta_{I(a-1,i)} \arrow{r}\arrow{d}\ & \Delta_{I(a-2,i)} \arrow{d} \arrow{ul}  \\
\Delta_{I(a-1,i+1)}\arrow{r} & \Delta_{I(a-2,i+1)}  \arrow{ul}\\
\end{tikzcd}
$$
The exchange ratios are equal to
$$
y_{I(a-1,i)}(V_2)=\frac{\Delta_{I(a-1,i-1)}(V_2)\Delta_{I(a-2,i+1)}(V_2)}{\Delta_{I(a-2,i)}(V_2)\Delta_{I(a-1,i+1)}(V_2)}
$$
so by Lemma \ref{lem: new minors} we get
$$
y_{I(a-1,i)}(V_2)=\left[\frac{\Delta_{I(a,i-1)}(V)\Delta_{I(a,i+1)}(V)}{\Delta_{I(a,i)}(V)\Delta_{I(a,i+1)}(V)}\right]^{-1}\cdot \frac{\Delta_{I(a-1,i-1)}(V)\Delta_{I(a-2,i+1)}(V)}{\Delta_{I(a-2,i)}(V)\Delta_{I(a-1,i+1)}(V)}=
$$
$$
\frac{\Delta_{I(a,i)}(V)\Delta_{I(a-1,i-1)}(V)\Delta_{I(a-2,i+1)}(V)}{\Delta_{I(a,i-1)}(V)\Delta_{I(a-2,i)}(V)\Delta_{I(a-1,i+1)}(V)}=y_{I(a-1,i)}(V).
$$
(d) Corner, $b=a-1,i=1$:
$$
\begin{tikzcd}
& \Delta_{I(a-1,1)} \arrow{r}\arrow{d} & \Delta_{I(a-2,1)}    \arrow{d}  \\
\boxed{\Delta_{I(a,k-1)}}\arrow{ur}& \Delta_{I(a-1,2)} \arrow{r} \ & \Delta_{I(a-2,2)} \arrow{ul}  
\end{tikzcd}
$$
Here we identify $\Delta_{I(a,k-1)}(V_2)$ with 
$\Delta_{I(a,k-1)}(V)=\Delta_{I(1,k-1)}(V_1)$ as above.
The exchange ratio is equal to
$$
y_{I(a-1,1)}(V_2)=\frac{\Delta_{I(a,k-1)}(V)\Delta_{I(a-2,2)}(V_2)}{\Delta_{I(a-2,1)}(V_2)\Delta_{I(a-1,2)}(V_2)}
$$
so by Lemma \ref{lem: new minors} we get
$$
\Delta_{I(a,k-1)}(V)\cdot \frac{\Delta_{I(a-2,2)}(V)\Delta_{I(a,k-1)}(V)}{\Delta_{I(a,2)}(V)}\cdot \frac{\Delta_{I(a,1)}(V)}{\Delta_{I(a-2,1)}(V)\Delta_{I(a,k-1)}(V)}\cdot 
\frac{\Delta_{I(a,2)}(V)}{\Delta_{I(a-1,2)}(V)\Delta_{I(a,k-1)}(V)}=
$$
$$
\frac{\Delta_{I(a-2,2)}(V)\Delta_{I(a,1)}(V)}{\Delta_{I(a-2,1)}(V)\Delta_{I(a-1,2)}(V)}=y_{I(a-1,1)}(V).
$$
\end{proof}

%This would be the main theorem:

%Freezing a column of minors  = splitting into two Scott-type quivers

%No change on the left, quasi-equivalence on the right. 

\section{Relation to braid varieties}
\label{sec: braids}

In this section we describe the map $\Phi$ in terms of braid varieties. We refer to \cite{CGGS1,CGGS2,CGGLSS,GLSBS} for more information and context and braid varieties, and only use some basic definitions. We will work on the variety of complete flags 
$$
\mathrm{Fl}_k=\{0=\CF_0\subset \CF_1\cdots \subset \CF_{k}=\C^k\},\ \dim \CF_i=i. 
$$

\begin{definition}
We say that two flags $\CF$ and $\CF'$ are in position $s_i$ if $\CF_j=\CF'_j$ for $j\neq i$ and $\CF_i\neq \CF'_i$. We will denote this by $\CF\xrightarrow{s_i}\CF'$.

We say that $\CF$ and $\CF'$ are in position $w_0$ if $\CF_{i}\oplus \CF'_{n-i}=\C^k$, in other words $\CF_i$ is transversal to $\CF'_{n-i}$ for all $i$.
\end{definition}

\begin{definition}
Given a braid $\beta=\sigma_{i_1}\cdots\sigma_{i_\ell}$, we define the braid variety as the space of sequences of flags
$$
\CF^{(0)}\xrightarrow{s_{i_1}}\CF^{(1)}\cdots \CF^{(\ell-1)}\xrightarrow{s_{i_\ell}}\CF_{(\ell)}
$$
such that $\CF^{(0)}$ is the standard flag and $\CF^{(\ell)}$ is the antistandard flag in $\C^k$:
$$
\CF^{(0)}_i=\langle e_1,\ldots,e_i\rangle,\ \CF^{(\ell)}_i=\langle e_{j-i+1},\ldots,e_k\rangle.
$$
\end{definition}
We will visualize the flags $\CF^{(j)}$ by labeling the regions in the braid diagram for $\beta$ by vector spaces such that each vertical cross-section provides a complete flag, see Figure \ref{fig:braid cut 3,8}. We recall an explicit construction \cite[Section 4]{CasalsGao} relating $\Pi_{k,n}^{\circ}$ to braid variety $X(\beta_{k,n})$ where 
$$\beta_{k,n}=(\sigma_1\dots\sigma_{k-1})^{n-k}(\sigma_1\dots\sigma_{k-1})\dots(\sigma_2\sigma_1)\sigma_1=(\sigma_1\dots\sigma_{k-1})^{n-k}w_0$$ (see also \cite{STWZ}). Here $T(k,n-k)=(\sigma_1\dots\sigma_{k-1})^{n-k}$ is the $(k,n-k)$ torus braid and $\sigma_1(\sigma_2\sigma_1)\cdots (\sigma_{k-1}\cdots \sigma_1)$ is the specific braid word for the half-twist braid.

%\textcolor{red}{Need to discuss. Is it true that $\Pi_{k,n}^o\simeq X(\beta)$ where $\beta=(\sigma_1\dots\sigma_{k-1})^{n-1}$? If so, what is a good reference? Confused about association of Pl\"uckers and chambers of the braid. }

Given a matrix $V=(v_1,\ldots,v_n)$, we can fill in the bottom row of the braid diagram for $\beta_{k,n}$ by the vectors $v_1,\ldots,v_n$. This uniquely determines the subspaces for all other regions as spans $\langle v_i,\ldots,v_j\rangle$ for appropriate $i,j$, see Figures \ref{fig:braid cut 3,8} and \ref{fig: splice braid}. The conditions $\Delta_{I(a,k)}(V)\neq 0$ are equivalent to the relative position conditions for each crossing of $\beta$. The conditions $\Delta_{I(1,i)}(V)\neq 0$ are equivalent to the fact that two flags
$$
\CF^{(0)}=\{0\subset \langle v_1\rangle \subset \langle v_1,v_2\rangle \subset\ldots \langle v_1,\ldots,v_{k}\rangle\}
$$
and 
$$
\CF^{(N)}=\{0\subset \langle v_n\rangle \subset \langle v_{n-1},v_n\rangle \subset\ldots \langle v_{n-k+1},\ldots,v_{n}\rangle\}
$$
are in position $w_0$. Therefore there is a unique matrix $M$ such that $M\CF^{(0)}$ is the standard flag and $M\CF^{(N)}$ is the antistandard flag.

Finally, the flags constructed as above determine the vectors $v_i$ only up to scalars. This can be fixed either by rescaling $v_i$, or by considering framed flags as in \cite{CasalsGao}. As a result, we obtain the following.

\begin{theorem}[\cite{CasalsGao,STWZ}]
Let $\Pi_{k,n}^{\circ,1}$ be the subset of $\Pi_{k,n}^{\circ}$ defined by 
$$
\Delta_{b,b+1,\ldots,b+k-1}=\Delta_{I(b,k)}=1.
$$
Then $X(\beta_{k,n})\simeq \Pi_{k,n}^{\circ,1}$.
\end{theorem}

\begin{figure}[ht!]
    \centering
    \input{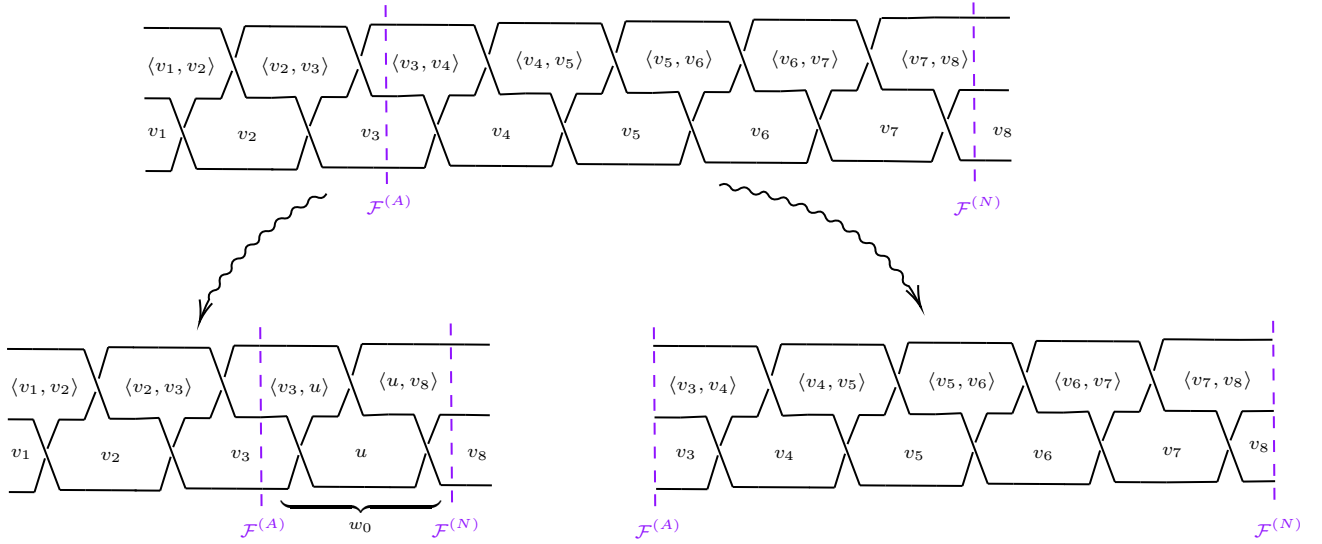}
    \caption{Freezing $\Delta_{348},\Delta_{378}$ in the braid associated to $X(\beta_{3,8})\simeq\Pi_{3,8}^{\circ,1}$.}
    \label{fig:braid cut 3,8}
\end{figure}

%\textcolor{red}{TODO: describe how we cut the braid, what is going on in the picture etc:
%$A=(a-1)(k-1)+$, $N=k(n-k)+\binom{n}{2}$ (should be $N=(n-k)(k-1)+\binom{k}{2}+1$ length of $w_0=\binom{k}{2}$, $k$ defines number of strands), cuts correspond to flags $\CF^{(A)}$ (check), $\CF^{(N)}$
%Claim: open subset $U_a$ says that these two flags are transversal
%$V_1$: all flags between them
%$V_2$: flags up to $\CF^{(A)}$, stitching them using $w_0$, braid $(\sigma_1\cdots \sigma_{k-1})^{a-1}w_0$.
%}

Let $\beta=(\sigma_1\dots\sigma_{k-1})^{n-k}(\sigma_1\dots\sigma_{k-1})\dots(\sigma_1\sigma_2)\sigma_1$ be the braid where $X(\beta)\simeq \Pi_{k,n}^{\circ,1}$. Then the process of freezing $\Delta_{I(a)}$ corresponds to severing the braid $\beta$ at flags $\CF^{(A)}$ and $\CF^{(N)}$ where $A=(a-1)(k-1)+1$ and $N=(n-k)(k-1)+\binom{k}{2}+1$. Upon severing the braid at the given flags we disassemble the braid into two separate braids decorated by the flags 
\begin{equation}
\label{eq: flag for V1}
\CF^{(A)}\xrightarrow{s_1}\CF^{(A+1)}\cdots\CF^{(N-1)}\xrightarrow{s_1}\CF^{(N)}
\end{equation}
and 
\begin{equation}
\label{eq: flag for V2}
\CF^{(0)}\xrightarrow{s_1}\CF^{(1)}\cdots\CF^{(A-1)}\xrightarrow{s_{k-1}}\CF^{(A)}\xrightarrow{s_1}\widetilde{\CF}^{(A+1)}\xrightarrow \cdots \widetilde{\CF}^{A+\binom{k}{2}-1}\xrightarrow{s_1}\CF^{(N)}.
\end{equation}
The first braid is decorated by the flags between $\CF^{(A)}$ and $\CF^{(N)}$ and is associated to $X(\beta_1)$ where $\beta_1=(\sigma_1\cdots \sigma_{k-1})^{n-k-a+1}w_0$. Note that the conditions defining the open subset $U_a$ guarantee that the flags $\CF^{(A)}$ and $\CF^{(N)}$ are in position $w_0$, so as above there is a unique matrix $M$ such that $M\CF^{(A)}$ is the standard flag and $M\CF^{(N)}$ is the antistandard flag.

Whereas for the second braid we splice together the flags $\CF^{(A)}$ and $\CF^{(N)}$ by interweaving the sequences of flags $\widetilde{\CF}$ associated to the half twist on $k$ strands, see Figure \ref{fig:braid cut 3,8} for an example of the  of the braid $\beta$ into its two separate components and Figure \ref{fig: splice braid} details the local splicing effect on the flags. This stitching of the half twist braid fills the bottom row with $k-2$ vectors $u_1,\dots, u_{k-2}$, and the intermediate flags $\widetilde{F}^{(A+j)}$ are uniquely determined by $\CF^{(A)}$ and $\CF^{(N)}$. Through this process the resulting braid is $\beta_2=(\sigma_1\cdots \sigma_{k-1})^{a-1}w_0$.

%\textcolor{red}{TO DO: Need to define a labeling for the decomposed flags.  Something about the choice of $A$ and how this guarantees that $w_0$ is included in the second braid - allowing for recovery of the smaller positroids. }
\begin{figure}
    \centering
    \tikzset{every picture/.style={line width=0.75pt}} %set default line width to 0.75pt        

\begin{tikzpicture}[x=0.75pt,y=0.75pt,yscale=-1,xscale=1]

%Straight Lines [id:da43081404952437174] 
\draw    (120.13,170.42) -- (180,170.42) ;
%Straight Lines [id:da860114993418754] 
\draw    (120.13,140.42) -- (210.25,139.92) ;
%Straight Lines [id:da440518013726922] 
\draw    (120.13,110.17) -- (240.25,110.42) ;
%Straight Lines [id:da521301808874939] 
\draw    (120.13,200.42) -- (150,200.17) ;
%Straight Lines [id:da7314439986107959] 
\draw    (119.88,230.17) -- (150,230.17) ;
%Straight Lines [id:da18667523098749994] 
\draw    (150,200.17) -- (170.25,229.92) ;
%Straight Lines [id:da8363625573031155] 
\draw    (150,230.17) -- (158.75,217.17) ;
%Straight Lines [id:da7814998766989167] 
\draw    (161.75,213.17) -- (169.75,200.42) ;
%Straight Lines [id:da8305055603193974] 
\draw    (180,170.42) -- (200.25,200.17) ;
%Straight Lines [id:da6774550780463704] 
\draw    (180,200.42) -- (188.75,187.42) ;
%Straight Lines [id:da6162696947327795] 
\draw    (191.75,183.42) -- (199.75,170.67) ;
%Straight Lines [id:da44884704622762994] 
\draw    (210.25,139.92) -- (230.5,169.67) ;
%Straight Lines [id:da7507289769076497] 
\draw    (210,170.67) -- (219,156.92) ;
%Straight Lines [id:da5750670591362803] 
\draw    (222,152.92) -- (230,140.17) ;
%Straight Lines [id:da3420912754354959] 
\draw    (240.25,110.42) -- (260.5,140.17) ;
%Straight Lines [id:da6303815720809545] 
\draw    (240.25,140.42) -- (249,127.42) ;
%Straight Lines [id:da00664484093595008] 
\draw    (252,123.42) -- (260,110.67) ;
%Straight Lines [id:da12293138358217881] 
\draw    (270.25,200.17) -- (290.5,229.92) ;
%Straight Lines [id:da3531186658767105] 
\draw    (270.25,230.17) -- (279,217.17) ;
%Straight Lines [id:da21541171624363797] 
\draw    (282,213.17) -- (290,200.42) ;
%Straight Lines [id:da8592888424023577] 
\draw    (300.25,170.17) -- (320.5,199.92) ;
%Straight Lines [id:da9538275072722171] 
\draw    (300.25,200.42) -- (309,187.17) ;
%Straight Lines [id:da1188496582775056] 
\draw    (312,183.17) -- (320,170.42) ;
%Straight Lines [id:da7470305306011495] 
\draw    (330.25,140.42) -- (350.5,170.17) ;
%Straight Lines [id:da4724737792643683] 
\draw    (330.25,170.42) -- (339,157.42) ;
%Straight Lines [id:da820264230580249] 
\draw    (342,153.42) -- (350,140.67) ;
%Straight Lines [id:da09561466808043906] 
\draw    (360.5,199.92) -- (380.75,229.67) ;
%Straight Lines [id:da41865092057924946] 
\draw    (360.5,229.92) -- (369.25,216.92) ;
%Straight Lines [id:da8009826831322315] 
\draw    (372.25,212.92) -- (380.25,200.17) ;
%Straight Lines [id:da10675803389716676] 
\draw    (390.25,170.42) -- (410.5,200.17) ;
%Straight Lines [id:da2490164329830915] 
\draw    (390.25,200.42) -- (399,187.42) ;
%Straight Lines [id:da4756099419509221] 
\draw    (402,183.42) -- (410,170.67) ;
%Straight Lines [id:da7775015032357266] 
\draw    (419.75,200.17) -- (440,229.92) ;
%Straight Lines [id:da5359178484754723] 
\draw    (419.75,230.17) -- (428.5,217.17) ;
%Straight Lines [id:da7013910998171433] 
\draw    (431.5,213.17) -- (439.5,200.42) ;
%Straight Lines [id:da9341299252924127] 
\draw    (169.75,200.42) -- (180,200.42) ;
%Straight Lines [id:da7271773704657327] 
\draw    (199.75,170.67) -- (210,170.67) ;
%Straight Lines [id:da2577040327200544] 
\draw    (230,140.17) -- (240.25,140.42) ;
%Straight Lines [id:da2622275519447739] 
\draw    (290,200.42) -- (300.25,200.42) ;
%Straight Lines [id:da15448483353228415] 
\draw    (320,170.42) -- (330.25,170.42) ;
%Straight Lines [id:da23631163115163267] 
\draw    (380.25,200.17) -- (390.5,200.17) ;
%Straight Lines [id:da8665178576273245] 
\draw    (410.5,200.17) -- (419.75,200.17) ;
%Straight Lines [id:da41608868579923364] 
\draw    (170.25,229.92) -- (270.25,230.17) ;
%Straight Lines [id:da450686844124742] 
\draw    (200.25,200.17) -- (270.25,200.17) ;
%Straight Lines [id:da41636902114130936] 
\draw    (230.5,169.67) -- (300.25,170.17) ;
%Straight Lines [id:da11959811607656179] 
\draw    (260.5,140.17) -- (330.25,140.42) ;
%Straight Lines [id:da057556342319579956] 
\draw    (290.5,229.92) -- (360.5,229.92) ;
%Straight Lines [id:da48805180011896043] 
\draw    (320.5,199.92) -- (360.5,199.92) ;
%Straight Lines [id:da42242196465294546] 
\draw    (380.75,229.67) -- (419.75,230.17) ;
%Straight Lines [id:da7302288135841646] 
\draw    (350.5,170.17) -- (390.25,170.42) ;
%Straight Lines [id:da6440989444748026] 
\draw    (440,229.92) -- (470.25,229.92) ;
%Straight Lines [id:da9620238841063735] 
\draw    (439.5,200.42) -- (469.75,200.42) ;
%Straight Lines [id:da38502237710150444] 
\draw    (410,170.67) -- (470.25,170.42) ;
%Straight Lines [id:da7841016819134863] 
\draw    (350,140.67) -- (470,140.67) ;
%Straight Lines [id:da7657279874013829] 
\draw    (260,110.67) -- (470,110.42) ;

% Text Node
\draw (129.17,209.01) node [anchor=north west][inner sep=0.75pt]  [font=\footnotesize] [align=left] {$\displaystyle v_{3}$};
% Text Node
\draw (444.83,209.34) node [anchor=north west][inner sep=0.75pt]  [font=\footnotesize] [align=left] {$\displaystyle v_{10}$};
% Text Node
\draw (394.5,209.34) node [anchor=north west][inner sep=0.75pt]  [font=\footnotesize] [align=left] {$\displaystyle u_{3}$};
% Text Node
\draw (318.83,209.34) node [anchor=north west][inner sep=0.75pt]  [font=\footnotesize] [align=left] {$\displaystyle u_{2}$};
% Text Node
\draw (208.83,209.01) node [anchor=north west][inner sep=0.75pt]  [font=\footnotesize] [align=left] {$\displaystyle u_{1}$};
% Text Node
\draw (122.83,179.34) node [anchor=north west][inner sep=0.75pt]  [font=\footnotesize] [align=left] {$\displaystyle \langle v_{3} ,v_{4} \rangle $};
% Text Node
\draw (418.5,177.67) node [anchor=north west][inner sep=0.75pt]  [font=\footnotesize] [align=left] {$\displaystyle \langle v_{9} ,v_{10} \rangle $};
% Text Node
\draw (123.5,148.01) node [anchor=north west][inner sep=0.75pt]  [font=\footnotesize] [align=left] {$\displaystyle \langle v_{3} ,v_{4} ,v_{5} \rangle $};
% Text Node
\draw (122.83,119.01) node [anchor=north west][inner sep=0.75pt]  [font=\footnotesize] [align=left] {$\displaystyle \langle v_{3} ,v_{4} ,v_{5} ,v_{6} \rangle $};
% Text Node
\draw (400.83,147.34) node [anchor=north west][inner sep=0.75pt]  [font=\footnotesize] [align=left] {$\displaystyle \langle v_{8} ,v_{9} ,v_{10} \rangle $};
% Text Node
\draw (381.5,117.34) node [anchor=north west][inner sep=0.75pt]  [font=\footnotesize] [align=left] {$\displaystyle \langle v_{7} ,v_{8} ,v_{9} ,v_{10} \rangle $};
% Text Node
\draw (223,178.67) node [anchor=north west][inner sep=0.75pt]  [font=\footnotesize] [align=left] {$\displaystyle \langle u_{1} ,u_{2} \rangle $};
% Text Node
\draw (337,178.67) node [anchor=north west][inner sep=0.75pt]  [font=\footnotesize] [align=left] {$\displaystyle \langle u_{2} ,u_{3} \rangle $};
% Text Node
\draw (249,147.67) node [anchor=north west][inner sep=0.75pt]  [font=\footnotesize] [align=left] {$\displaystyle \langle u_{1} ,u_{2} ,u_{3} \rangle $};

\end{tikzpicture}
    \caption{Braid diagram and flags for Example \ref{ex: 5,10}. Here $\langle u_1\rangle =\langle v_3,v_4\rangle \cap \langle v_7,v_8,v_9,v_{10}\rangle$,
    $\langle u_2\rangle =\langle v_3,v_4,v_5\rangle \cap \langle v_8,v_9,v_{10}\rangle$ and 
    $\langle u_3\rangle =\langle v_3,v_4,v_5,v_6\rangle \cap \langle v_9,v_{10}\rangle$.}
    \label{fig: splice braid}
\end{figure}
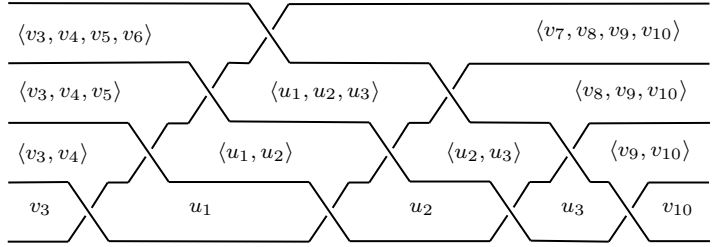

Finally, we can compare the cluster structures on braid varieties. The cluster structure on $\Pi_{k,n}^{\circ,1}$ is obtained from  \eqref{eq: scott} by removing the frozen variables $\Delta_{I(b,k)}$ from $Q_V$. 

\begin{theorem}
The map $\overline{\Phi}_a:V\mapsto (V_1,\overline{V_2})$,
$$
V_1=(v_a,\ldots,v_n),\quad \overline{V_2}=\left(v_1,\ldots,v_a,u_1,\ldots,u_{k-2},\frac{v_n}{\Delta_{I(a,k-1)}}\right)
$$
defines a quasi-cluster isomorphism
between $U_a^1=U_a\cap \Pi_{k,n}^{\circ,1}$ and $\Pi_{k,n-a+1}^{\circ,1}\times \Pi_{k,a+k-1}^{\circ,1}$.
\end{theorem}

\bibliographystyle{amsplain}

\end{document}